\documentclass[twoside,11pt,reqno]{article}
\usepackage{jmlr2e}
\usepackage[export]{adjustbox}

\usepackage{amsmath,ntheorem,amssymb,enumitem,natbib,color,ifthen,graphicx,hyperref}
\usepackage{xargs}
\usepackage{todonotes}
\usepackage[utf8]{inputenc}

\usepackage{aliascnt,bbm}

\def\nset{{\mathbb{N}}}
\def\rset{\mathbb R}
\def\zset{\mathbb Z}

\def\eqsp{\;}

\newcommand{\eg}{e.g.}
\newcommand{\ie}{i.e.}
\newcommand{\wrt}{with respect to}
\newcommand{\as}{\text{a.s.}}
 \newcommand{\pscal}[2]{\left\langle#1,#2\right\rangle}
\newcommand{\normW}[2]{\left|#1\right|_#2}
\newcommand{\normWm}[2]{\left\|#1\right\|_#2}
\newcommand{\normLq}[2]{\left\|#1\right\|_{L^{#2}}}

\newcommand{\un}{\ensuremath{\mathbbm{1}}}
\newcommand{\eqdef}{\ensuremath{\stackrel{\mathrm{def}}{=}}}

\renewcommand{\H}[2]{\ifthenelse{\equal{#2}{}}{H_{#1}}{H_{#1}(#2)}}
\newcommand{\hatH}[2]{\ifthenelse{\equal{#2}{}}{\widehat{H}_{#1}}{\widehat{H}_{#1}(#2)}}

\def\Xset{\mathsf{X}} % Espace d 'etat
\def\Xsigma{\mathcal{X}} % tribu sur X
\def\F{\mathcal{F}} % filtration
\def\N{\mathcal{N}}
\def\M{\mathcal{M}}

\def\true{\mathsf{true}}
\newcommandx\sequence[3][2=t,3=\zset]
{\ifthenelse{\equal{#3}{}}{\ensuremath{\{ #1_{#2}\}}}{\ensuremath{\{ #1_{#2}, \eqsp #2 \in #3 \}}}}

\def\Am{\mathsf{A}}

\def\PP{\mathbb{P}} % proba
\newcommand{\CPP}[3][]
{\ifthenelse{\equal{#1}{}}{{\mathbb P}\left(\left. #2 \, \right| #3 \right)}{{\mathbb P}_{#1}\left(\left. #2 \, \right | #3 \right)}}
\def\PE{\mathbb{E}} % esperance
\newcommand{\CPE}[3][]
{\ifthenelse{\equal{#1}{}}{{\mathbb E}\left[\left. #2 \, \right| #3 \right]}{{\mathbb E}_{#1}\left[\left. #2 \, \right | #3 \right]}}

\def\L{\mathcal{L}} % espace des fonctions

\newcommand{\norm}[1]{\left\Vert#1\right\Vert}

\def\Prox{\operatorname{Prox}}

\def\true{\mathsf{true}}
\usepackage{bm}%bold matH

\newcommand{\pma}[1]{{\color{black} #1}}

\theoremstyle{plain}
\newtheorem{theorem}{Theorem}
\newtheorem{assumption}{H\hspace{-3pt}}

\newaliascnt{proposition}{theorem}
\newtheorem{proposition}[proposition]{Proposition}
\aliascntresetthe{proposition}

\newaliascnt{lemma}{theorem}
\newtheorem{lemma}[lemma]{Lemma}
\aliascntresetthe{lemma}
\newaliascnt{corollary}{theorem}

\aliascntresetthe{corollary}

\newaliascnt{definition}{theorem}

\aliascntresetthe{definition}

\newtheorem{algorithm}{Algorithm}

\newaliascnt{remark}{theorem}

\aliascntresetthe{remark}

\newcommand{\coint}[1]{\left[#1\right)}
\newcommand{\ocint}[1]{\left(#1\right]}
\newcommand{\ooint}[1]{\left(#1\right)}
\newcommand{\ccint}[1]{\left[#1\right]}

\def\rmd{\mathrm{d}}
\def\rme{\mathrm{e}}
\def\argmin{\operatorname{argmin}}

\def\1{\mathbbm{1}}
\def\gauss{\operatorname{N}}
\def\bU{\mathbf{U}}
\def\bY{\mathbf{Y}}
\def\bu{\mathbf{u}}

\def\bw{\mathbf{w}}
\def\tpi{\tilde{\pi}}
\def\iid{i.i.d.}
\def\bpi{\bar{\pi}}
\def\pg{\mathsf{PG}}

\newcommand{\Xb}[2]{X_{#1}^{(#2)}}

\begin{document}
\title{On Perturbed Proximal Gradient  Algorithms}

\author{\name Yves F. Atchad\'e \addr University
  of Michigan, 1085 South University, Ann Arbor, 48109, MI, United
  States, \email yvesa@umich.edu   \AND
  \name Gersende Fort  \addr  LTCI, CNRS, Telecom ParisTech, Université Paris-Saclay. 46, rue Barrault 75013 Paris, France, \email gersende.fort@telecom-paristech.fr
  \AND  Eric Moulines  \addr LTCI, CNRS, Telecom ParisTech, Université Paris-Saclay. 46, rue Barrault 75013 Paris, France, \email eric.moulines@polytechnique.edu }

\editor{Léon Bottou}
\maketitle

\begin{abstract}
  We study a version of the proximal gradient algorithm for which the gradient
  is intractable and is approximated by Monte Carlo methods (and in particular
  Markov Chain Monte Carlo). We derive conditions on the step size and the
  Monte Carlo batch size under which convergence is guaranteed: both increasing
  batch size and constant batch size are considered. We also derive non-asymptotic
  bounds for an averaged version.  Our results cover both the cases of biased
  and unbiased Monte Carlo approximation. To support our findings, we discuss
  the inference of a sparse generalized linear model with random effect and the
  problem of learning the edge structure and parameters of sparse undirected
  graphical models.
\end{abstract}

\begin{keywords}
  Proximal Gradient Methods; Stochastic Optimization; Monte Carlo
approximations; Perturbed Majorization-Minimization algorithms.
\end{keywords}

\section{Introduction}
\label{sec:intro}

This paper deals with statistical optimization problems of the form:
\begin{equation*}
  \text{\textbf{(P)}}      \quad \min_{\theta\in \rset^d} F(\theta) \qquad \text{with $F = f + g$}   \eqsp.
\end{equation*}
This problem occurs in a variety of statistical and machine learning problems,
where $f$ is a measure of fit depending implicitly on some observed data and
$g$ is a regularization term that imposes structure to the solution.
Typically, \pma{$f$ is a differentiable function with a Lipschitz gradient}, whereas $g$ might be non-smooth (typical examples include sparsity
inducing penalty).

\begin{assumption}
  \label{A1} The function
  $g:\;\rset^d\to \ccint{0,+\infty}$ is convex, not identically
  $+\infty$, and lower semi-continuous. The function $f: \rset^d\to
  \rset$ is convex, continuously differentiable on $\rset^d$ and there exists a finite
  non-negative constant $L$ such that, for all $\theta,\theta'\in\rset^d$,
  \[\|\nabla f(\theta)-\nabla f(\theta')\| \leq L\|\theta-\theta'\| \;,\]
  where $\nabla f$ denotes the gradient of $f$.
\end{assumption}
We denote by $\Theta$ the domain of $g$: $\Theta \eqdef \{\theta \in \rset^d:
g(\theta) < \infty \}$.
\begin{assumption}
  \label{A1:compl} The set  $\mathrm{argmin}_{\theta\in\Theta} F(\theta)$ is a non empty subset of $\Theta$.
\end{assumption}
In this paper, we focus on the case where $f +g$ and $\nabla f$ are both
intractable.  This setting has not been widely considered despite the
considerable importance of such models in statistics and machine learning.
Intractable likelihood problems naturally occur for example in inference for bayesian networks
(\eg\ learning the edge structure and the parameters in an undirected graphical
models), regression with latent variables or random effets, missing data,
etc... In such applications, $f$ is the negated log-likelihood of a conditional
Gibbs measure $\pi_\theta$ known only up to a normalization constant and the
gradient of $\nabla f(\theta)$ is typically expressed as a very
high-dimensional integral w.r.t. the associated Gibbs measure $\nabla
f(\theta)= \int H_\theta(x) \pi_\theta(\rmd x)$.  Of course, this integral
cannot be computed in closed form and should be approximated. Most often, some
forms of Monte Carlo integration (such as Markov Chain Monte Carlo, or MCMC) is
the only option.

To cope with problems where $f+g$ is intractable and possibly
non-smooth, various methods have been proposed.  Some of these works
focused on stochastic sub-gradient and mirror descent algorithms;
see~\cite{nemirovski:judistky:lan:shapiro:2008,duchietal11,NIPS2011_0942,lan12,juditsky:nemirovski:2011a,juditsky:nemirovski:2011b}.
Other authors have proposed algorithms based on proximal operators to
better exploit the smoothness of $f$ and the properties of $g$ (see
e.g. \pma{\cite{combettes:wajs:2005}}; \cite{hu:kwok:pan:2009,xiao10,juditsky:nemirovski:2011a,juditsky:nemirovski:2011b}).

The current paper focuses on the proximal gradient algorithm \pma{(see e.g.~\cite{beck:teboulle:2010,pesquet:combettes:2012,parikh:boyd:2013} for literature review and further references).} The proximal map
\pma{(\cite{moreau:1962})} associated to $g$ is defined for
$\gamma>0$ and $\theta \in \rset^d$ by:
\begin{equation}
  \label{eq:definition-proximal-operator}
  \Prox_{\gamma,g}(\theta) \eqdef \argmin_{\vartheta\in\Theta} \left\{g(\vartheta)+\frac{1}{2\gamma}\|\vartheta-\theta\|^2\right\}.
\end{equation}
Note that under \autoref{A1}, there exists an unique point $\vartheta$
minimizing the RHS of (\ref{eq:definition-proximal-operator}) for any $\theta
\in \rset^d$ and $\gamma>0$. The proximal gradient algorithm is an iterative algorithm
which, given an initial value $\theta_0 \in \Theta$ and a sequence of positive
step sizes $\{\gamma_n, n \in \nset \}$, produces a sequence of parameters
$\{\theta_n, n \in \nset \}$ as follows:
\begin{algorithm}[Proximal gradient algorithm]
  \label{proxalgo}  Given
  $\theta_n$, compute
  \begin{equation}
    \label{eq:proxalgo:recursion}
    \theta_{n+1}=\Prox_{\gamma_{n+1}, g}\left(\theta_n -\gamma_{n+1}\nabla f(\theta_n)\right) \eqsp.
  \end{equation}
\end{algorithm}
When $\gamma_n = \gamma$ for any $n$, it is known that the iterates of the proximal gradient algorithm
$\{\theta_n, n \in \nset \}$ (\autoref{proxalgo}) converges to
$\theta_\infty$, this point is a fixed point of the proximal-gradient map
\begin{equation}
  \label{eq:definition:ProximalMap}
  T_\gamma(\theta) \eqdef \Prox_{\gamma, g} \left(\theta-\gamma\nabla f(\theta)\right) \eqsp.
\end{equation}
Under \autoref{A1} and \autoref{A1:compl}, when $\gamma_n \ocint{0,2/L}$ and
$\inf_n \gamma_n>0$, it is indeed known that the iterates of the proximal gradient algorithm
$\sequence{\theta}[n][\nset]$ defined in \eqref{eq:proxalgo:recursion} converges to a point in the set $\L$ of the
solutions of $(P)$ which coincides with the fixed points of the mapping $T_\gamma$ for any $\gamma \in \ooint{0,2/L}$
\begin{equation}
\label{eq:definition:Lset}
\L  \eqdef\argmin_{\theta \in \Theta} F(\theta) = \{\theta \in \Theta: \theta = T_{\gamma}(\theta) \} \eqsp.
\end{equation}
(see e.g.  \cite[Theorem~3.4. and Proposition~3.1.(iii)]{combettes:wajs:2005}).

Since $\nabla f(\theta)$ is intractable, the gradient $\nabla f(\theta_n)$ at
$n$-th iteration is replaced by an approximation $H_{n+1}$:
\begin{algorithm}[Perturbed Proximal Gradient algorithm]
  \label{sto_proxalgo}
  Let $\theta_0\in \Theta$ be the initial solution and
  $\sequence{\gamma}[n][\nset]$ be a sequence of positive step sizes. For
  $n\geq 1$, given $(\theta_0,\dots, \theta_n)$ construct an approximation  $H_{n+1}$ of $\nabla f(\theta_n)$
  and compute
\begin{equation}
\label{eq:perturbedalgo:recursion}
\theta_{n+1}=\Prox_{\gamma_{n+1},g}\left(\theta_n - \gamma_{n+1} H_{n+1}\right) \eqsp.
\end{equation}
\end{algorithm}
We provide in \autoref{prop:CvgCvx:PerturbedProximalGradient} sufficient
conditions on the perturbation $\eta_{n+1} = H_{n+1} - \nabla f(\theta_n)$ to
obtain the convergence of the perturbed proximal gradient sequence given by
\eqref{eq:perturbedalgo:recursion}.  We then consider an averaging scheme of
the perturbed proximal gradient algorithm: given non-negative weights
$\sequence{a}[n][\nset]$, \autoref{theo:rate-convergence-basic} provides
non-asymptotic bound of the deviation between $\sum_{k=1}^n a_k F(\theta_k) /
\sum_{k=1}^n a_k$ and the minimum of $F$.  Our results complement and extend
\cite{Rosasco:Villa:Vu:2014,nitanda:2014,xiao:zhang:2014}.

We then consider the case where the gradient $\nabla f(\theta)= \int_\Xset
H_\theta(x)\pi_\theta(\rmd x)$ is defined as an expectation (see \autoref{A2}
in \autoref{sec:MonteCarloProxGdt}).  In this case, at each iteration $\nabla
f(\theta_n)$ is approximated by a Monte Carlo average $H_{n+1}= m_{n+1}^{-1}
\sum_{j=1}^{m_{n+1}} H_{\theta_n}(X_{n+1}^{(j)})$ where $m_{n+1}$ is the size
of the Monte Carlo batch and $\{X_{n+1}^{(j)}, 1 \leq j \leq m_{n+1} \}$ is the
Monte Carlo batch.  Two different settings are covered.  In the first setting,
the samples $\{ X_{n+1}^{(j)}, 1 \leq j \leq m_{n+1}\}$ are conditionally
independent and identically distributed (i.i.d.)  with distribution
$\pi_{\theta_n}$. In such case, the conditional expectation of $H_{n+1}$ given
all the past iterations, denoted by $\CPE{H_{n+1}}{\F_n}$ (see
\autoref{sec:MonteCarloProxGdt}), is equal to $\nabla f(\theta_n)$.  In the
second setting, the Monte Carlo batch $\{X_{n+1}^{(j)}, 1 \leq j \leq
m_{n+1}\}$ is produced by running a MCMC algorithm. In such case, the
conditional distribution of $X_{n+1}^{(j)}$ given the past is no longer exactly
equal to $\pi_{\theta_n}$ which implies that $\CPE{H_{n+1}}{\F_n} \ne \nabla
f(\theta_n)$.

  \autoref{theo:approxsto:cvg} (resp. \autoref{theo:IncreasingBatch}) establish
  the convergence of the sequence $\{\theta_n, n\in \nset \}$ when the batch
  size $m_n$ is either fixed or increases with the number of iterations $n$.
  When the Monte Carlo batch $\{ X_{n+1}^{(j)}, 1 \leq j \leq m_{n+1}\}$ is  \iid\
  conditionally to the past the two theorems essentially say that with probability one,
  $\{\theta_n, n \in \nset \}$ converges to an element of the set of minimizer $\L$ as
  soon as $\sum_n \gamma_n = + \infty$ and $\sum_n \gamma_{n+1}^2/m_{n+1} <
  \infty$. Hence, one can choose either a fixed step size $\gamma_n = \gamma$
  and a batch size $\{m_n, n \in \nset \}$ increasing at least linearly (up to
  a logarithmic factor); or a decreasing step size and a fixed batch size $m_n
  = m$.  When $\{ X_{n+1}^{(j)}, 1 \leq j \leq m_{n+1}\}$ is produced by a MCMC
  algorithm (under appropriate assumptions) our theorems essentially say that
  the same convergence result holds if $\sum_n \gamma_{n} = \infty$ and $\sum_n
  \gamma_{n+1}^2 < \infty$ when $m_n=m$ is constant across iterations or
  $\sum_n \gamma_{n+1}/m_{n+1} < \infty$ if the batch size is increased.

  \autoref{theo:approxsto:cvg} and \autoref{theo:IncreasingBatch} also provide
  non asymptotic bounds for the difference $\Delta_n= \sum_{k=1}^n a_k
  F(\theta_k) / \sum_{k=1}^n a_k - \min F$ in $L^q$-norm for $q \geq 1$.  When
  the batch size sequence $m_{n+1}$ increases linearly at each iteration while
  the step size $\gamma_{n+1}$ is held constant, $\Delta_n = O(\ln n /n)$.  We
  recover (up to a logarithmic factor) the rate of the proximal gradient
  algorithm. If we now compare the complexity of the algorithms in terms of the
  number of simulations $N$ needed (and not the number of iterations), the
  error bound decreases like $O(N^{-1/2})$. The same error bound can be
  achieved by choosing a fixed batch size and a decreasing step size $\gamma_n
  = O(1/\sqrt{n})$.

  In \autoref{sec:ex2}, these results are illustrated with the problem of
  estimating a high-dimensional discrete graphical models. In \autoref{sec:ex1}, 
  we consider high-dimensional random effect logistic regression model.
  All the proofs are postponed to \autoref{sec:proofs}.

\section{Perturbed proximal gradient algorithms}
\label{sec:stochastic-proximal-gradient}

The key property to study the behavior of the sequence the perturbed
proximal gradient algorithm is the following elementary lemma which
might be seen as a deterministic version of the Robbins-Siegmund lemma
(see e.g.~\cite[Lemma 11, Chapter 2]{polyak87}). It replaces in our
analysis \cite[Lemma 3.1]{combettes:2001} for quasi-Fejer sequences
and modified Fejer monotone sequences
(see~\cite{lin:rosasco:villa:zhou:2015}). \pma{Compared to the
  Robbins-Siegmund Lemma, the sequence $(\xi_n)_n$ is not assumed to
  be nonnegative. When applied in the stochastic context as in
  Section~\ref{sec:MonteCarloProxGdt}, the fact that the result is purely deterministic and
  deals with signed perturbations $\xi_n$ allows more flexibility in the study
  of the dynamics.}  
\begin{lemma}\label{lemma:RobbinsSiegmund:deterministe}
  Let $\sequence{v}[n][\nset]$ and $\sequence{\chi}[n][\nset]$ be non-negative
  sequences and $\sequence{\xi}[n][\nset]$ be such that $\sum_n \xi_n$
  exists.  If for any $n \geq 0$,
  \[
  v_{n+1}\leq v_n - \chi_{n+1} + \xi_{n+1}
  \]
  then $\sum_n \chi_n < \infty$ and $\lim_n v_n$ exists.
\end{lemma}
\begin{proof}
  See \autoref{sec:proof:lemma:RobbinsSiegmund:deterministe}
\end{proof}
Applied with $v_n = \|\theta_n - \theta_\star\|$ for some $\theta_\star \in
\L$, this lemma is the key result for the proof of the following theorem, which
provides sufficient conditions on the stepsize sequence
$\sequence{\gamma}[n][\nset]$ and on the \emph{approximation error}:
\begin{equation}
  \label{eq:definition-eta}
  \eta_{n+1}\eqdef H_{n+1}-\nabla f(\theta_n) \eqsp,
\end{equation}
for the sequence $\sequence{\theta}[n][\nset]$ to converge to a point
$\theta_\infty$ in the set $\L$ of the minimizers of $F$. Denote by
$\pscal{\cdot}{\cdot}$ the usual inner product on $\rset^d$ associated to the
norm $\| \cdot \|$.
\begin{theorem} \label{prop:CvgCvx:PerturbedProximalGradient}
  Assume \autoref{A1} and \autoref{A1:compl}. Let
  $\sequence{\theta}[n][\nset]$ be given by \autoref{sto_proxalgo} with
  step sizes satisfying $\gamma_n \in \ocint{0, 1/L}$ for any $n\geq 1$ and
  $\sum_n \gamma_n = +\infty$. If the following series converge
\begin{equation}\label{eq:CS:BorelCantelli}
\sum_{n \geq 0} \gamma_{n+1}   \pscal{T_{\gamma_{n+1}}(\theta_n)}{\eta_{n+1}} \eqsp,  \, \quad
\sum_{n \geq 0} \gamma_{n+1} \eta_{n+1} \eqsp, \quad  \sum_{n \geq 0} \gamma_{n+1}^2 \|\eta_{n+1}\|^2  \eqsp,
\end{equation}
then there exists $\theta_\infty \in \L$ such that $\lim_n \theta_n =
\theta_\infty$.
\end{theorem}
\begin{proof}
  See \autoref{sec:proof:prop:CvgCvx:PerturbedProximalGradient}.
\end{proof}
\autoref{prop:CvgCvx:PerturbedProximalGradient} applied with $\eta_{n+1} =0$
provides sufficient conditions for the convergence of \autoref{proxalgo} to
$\L$: the algorithm converges as soon as $\gamma_n \in \ocint{0,1/L}$ and
$\sum_n \gamma_n = + \infty$.

Sufficient conditions for the convergence of $\sequence{
  \theta}[n][\nset]$ are also provided in~\cite{combettes:wajs:2005}.
When applied to our settings, Theorem 3.4. in
\cite{combettes:wajs:2005} require $\sum_n \|\eta_{n+1} \|<\infty$ and
$\inf_n \gamma_n >0$, \pma{which for instance cannot accommodate the
  fixed Monte Carlo batch size stochastic algorithms considered in
  this paper. The same limitation applies to the analysis of the
  stochastic quasi-Fejer iterations (see
  \cite{combettes:pesquet:2014}) which in our particular case requires
  $\sum_n \gamma_{n+1} \|\eta_{n+1}\|<\infty$.  These conditions are
  weakened in~\autoref{prop:CvgCvx:PerturbedProximalGradient}. However
  in all fairness we should mention that unlike the present work,
  \cite{combettes:wajs:2005} and \cite{combettes:pesquet:2014} deal
  with infinite-dimensional problems which raises additional technical
  difficulties, and study algorithms that include a relaxation
  parameter. Furthermore, in the case where $\eta_n\equiv 0$, larger
  values of the stepsize $\gamma_n$ are allowed ($\gamma_n \in (0,2/L]$.)}

\bigskip

Let $\{a_0, \cdots, a_n\}$ be non-negative real numbers.
\autoref{theo:rate-convergence-basic} provides a control of the weighted sum
$\sum_{k=1}^n a_k ( F(\theta_k) - \min F)$.
\begin{theorem}
  \label{theo:rate-convergence-basic}
  Assume \autoref{A1} and \autoref{A1:compl}.  Let
  $\sequence{\theta}[n][\nset]$ be given by \autoref{sto_proxalgo} with
  $\gamma_n \in \ocint{0, 1/L}$ for any $n\geq 1$.  For any non-negative
  weights $\{a_0, \cdots, a_n \}$, any $\theta_\star \in \L$ and any $n \geq
  1$,
\[
\sum_{k=1}^n a_k \left\{ F(\theta_k) - \min F \right\} \leq U_n(\theta_\star)
\]
where $T_\gamma$ and $\eta_n$ are given by (\ref{eq:definition:ProximalMap}) and (\ref{eq:definition-eta}) respectively and
  \begin{multline}
    U_n(\theta_\star) \eqdef \frac{1}{2}\sum_{k=1}^n
    \left(\frac{a_k}{\gamma_k}-\frac{a_{k-1}}{\gamma_{k-1}}\right)\|\theta_{k-1}-\theta_\star\|^2
    + \frac{a_0}{2 \gamma_0} \| \theta_0 - \theta_\star \|^2  \\
    - \sum_{k=1}^{n} a_{k} \pscal{T_{\gamma_{k}}(\theta_{k-1}) -
      \theta_\star}{\eta_k} + \sum_{k=1}^n a_k \gamma_k \|\eta_k \|^2 \eqsp.
    \label{eq:rate-convergence-basic}
  \end{multline}
\end{theorem}
\begin{proof}
  See~\autoref{sec:prooflemma:rate-convergence-basic}.
\end{proof}
%\autoref{theo:rate-convergence-basic} provides a convergence rate in a quite
%general setting: the step sizes $\sequence{\gamma}[n][\nset]$ are not
%necessarily constant and the weights $\{a_0, \cdots, a_n \}$ are non-negative
%but otherwise arbitrary.
When applied with $\eta_n = 0$, \autoref{theo:rate-convergence-basic} gives an
explicit bound of the difference $\Delta_n= A_n^{-1} \sum_{j=1}^n a_j
F(\theta_j) - \min F$ where $A_n= \sum_{k=1}^n a_k$ for the (exact) proximal
gradient sequence $\sequence{\theta}[n][\nset]$ given by \autoref{proxalgo}.
When the sequence $\{a_n/\gamma_n, n\geq 1 \}$ is non decreasing,
\eqref{eq:rate-convergence-basic} shows that $\Delta_n= O(a_n A_n^{-1}
\gamma_n^{-1})$.

Taking $a_k=1$ for any $k \geq 0$ provides a bound for the cumulative regret.
When $a_k =1$, $\gamma_k = 1/L$ for any $k \geq 0$, \cite[Proposition
1]{schmidt:leroux:bach:2011} provides a bound of order $O(1)$ under the
assumption that $\sum_n \|\eta_{n+1}\|<\infty$.  Using the inequality
$|\pscal{T_{1/L}(\theta_{k}) - \theta_\star}{\eta_{k+1}}| \leq \|\theta_k -
\theta_\star \| \|\eta_{k+1}\|$
(see~\autoref{lem:Lipschitz:GradientProximalMap}), the upper bound
$U_n(\theta_\star)$ in (\ref{eq:rate-convergence-basic}) is also $O(1)$.

When $a_n = \gamma_n$ for any $n \geq 0$, then $\sup_n U_n(\theta_\star) <
\infty$ under the assumptions that the series
\[
\sum_n \gamma_n \pscal{T_{\gamma_n}(\theta_{n-1}) - \theta_\star}{\eta_n}
\eqsp, \qquad \sum_n \gamma_{n}^2 \| \eta_{n} \|^2 \eqsp, \] converge.  In
this case,  we have
\[
\left(\frac{\sum_{k=1}^n \gamma_k F(\theta_k)}{\sum_{k=1}^n \gamma_k} - \min F
\right)= O\left(\left( \sum_{k=1}^n \gamma_k \right)^{-1} \right).
\]

\bigskip

Consider the weighted averaged sequence $\sequence{\bar\theta}[n][\nset]$
defined by
  \begin{equation}\label{ave_est}
    \bar\theta_n\eqdef  \frac{1}{A_n}\sum_{k=1}^n a_k \theta_k  \eqsp.
  \end{equation}
  Under \autoref{A1} and \autoref{A1:compl}, $F$ is convex so that $
  F\left(\bar \theta_n \right) \leq A_n^{-1} \sum_{k=1}^n a_k F(\theta_k)$.
  Therefore, \autoref{theo:rate-convergence-basic} also provides convergence
  rates for $F(\bar \theta_n) - \min F$.

\section{Stochastic Proximal Gradient algorithm}
\label{sec:MonteCarloProxGdt}
In this section, it is assumed that $H_{n+1}$ is a Monte Carlo approximation of $\nabla f(\theta_n)$, where $\nabla f(\theta)$ satisfies the
following assumption:
\begin{assumption}
  \label{A2}
  for all $\theta\in\Theta$,
  \begin{equation}
    \label{grad}
    \nabla f(\theta)=\int_\Xset H_\theta(x)\pi_\theta(\rmd x) \eqsp,
  \end{equation}
  for some probability measure $\pi_\theta$ on a measurable space
  $(\Xset,\Xsigma)$ and an integrable function $(\theta,x) \mapsto H_\theta(x)$
  from $\Theta \times \Xset$ to $\Theta$.
\end{assumption}
\pma{Note that $\Xset$ is not necessarily a topological space, even
  if, in many applications, $\Xset \subseteq \rset^d$.}

Assumption \autoref{A2} holds in many problems (see \autoref{sec:ex2} and
\autoref{sec:ex1}).  To approximate $\nabla f(\theta)$, several options are
available. Of course, when the dimension of the state space $\Xset$ is small to
moderate, it is always possible to perform a numerical integration using either
Gaussian quadratures or low-discrepancy sequences. Another possibility is to
approximate these integrals: nested Laplace approximations have been considered
recently for example in \cite{buhlmannetal:2014} and further developed in
\cite{ogden2015}. Such approximations necessarily introduce some
bias, which might be difficult to control. In addition, these techniques are
not applicable when the dimension of the state space $\Xset$ becomes large. In
this paper, we rather consider some form of Monte Carlo approximation.

When sampling $\pi_\theta$ is doable, then an obvious choice is to
use a naive Monte Carlo estimator which amounts to sample a batch $\{
\Xb{n+1}{j}, 1 \leq j \leq m_{n+1} \}$ independently of the past values of the
parameters $\{\theta_j, j \leq n\}$ and of the past draws i.e.  independently
of the $\sigma$-algebra
\begin{equation}
\label{eq:definition-Fn}
\F_n \eqdef \sigma(\theta_0, \Xb{k}{j}, 0 \leq k \leq n, 0 \leq j \leq m_k) \eqsp.
\end{equation}
We then form
\[
H_{n+1}= m_{n+1}^{-1} \sum_{j=1}^{m_{n+1}} H_{\theta_n}(\Xb{n+1}{j}) \eqsp.
\]
Conditionally to $\F_n$, $H_{n+1}$ is an unbiased estimator of $\nabla
f(\theta_n)$.  The batch size $m_{n+1}$ can either be chosen to be fixed
across iterations or to increase with $n$ at a certain rate.  In the first
case, $H_{n+1}$ is not converging.  In the second case, the approximation error
is vanishing. The fixed batch-size case is closely related to Robbins-Monro
stochastic approximation (the mitigation of the error is performed by letting
the stepsize $\gamma_n \to 0$); the increasing batch-size case is related to
Monte Carlo assisted optimisation; see for example \cite{geyer:1994}.

The situation that we are facing in \autoref{sec:ex2} and \autoref{sec:ex1} is
more complicated because direct sampling from $\pi_\theta$ is not an option.
Nevertheless, it is fairly easy to construct a Markov kernel $P_\theta$ with
invariant distribution $\pi_\theta$. Monte Carlo Markov Chains (MCMC) provide a
set of principled tools to sample from complex distributions over large
dimensional spaces.  In such case, conditional to the past, $\{ \Xb{n+1}{j}, 1
\leq j \leq m_{n+1} \}$ is a realisation of a Markov chain with transition
kernel $P_{\theta_n}$ and \pma{started from} $\Xb{n}{m_n}$ (the last sample
draws in the previous minibatch).

\pma{Recall that a Markov kernel $P$ is an application on $\Xset
  \times \Xsigma$, taking values in $\ccint{0,1}$ such that for any $x
  \in \Xset$, $P(x,\cdot)$ is a probability measure on $\Xsigma$; and
  for any $A \in \Xsigma$, $x \mapsto P(x,A)$ is measurable.
  Furthermore, if $P$ is a Markov kernel on $\Xset$, we denote by
  $P^k$ the $k$-th iterate of $P$ defined recursively as $P^0(x,A)
  \eqdef \1_A(x)$, and $P^k(x,A) \eqdef \int P^{k-1}(x,\rmd z)P(z,A)$,
  $k\geq 1$. Finally, the kernel $P$ acts on probability measure: for
  any probability measure $\mu$ on $\Xsigma$, $\mu P$ is a probability
  measure defined by
  \[
\mu P(A) \eqdef \int \mu(\rmd x) P(x,A), \qquad A \in \Xsigma;
\]
and $P$ acts on positive measurable functions: for a measurable
function $f:\Xset \to \rset_+$, $Pf$ is a function defined by
\[
Pf(x) \eqdef \int f(y) \, P(x, \rmd y).
\]
We refer the reader to \cite{meyn:tweedie:2009} for the definitions
and basic properties of Markov chains.}

\pma{In this Markovian setting,} it is possible to consider the fixed batch
case and the increasing batch case. From a mathematical standpoint,
the fixed batch case is trickier, because $H_{n+1}$ is no longer an
unbiased estimator of $\nabla f(\theta_n)$, \ie\ the bias $B_n$
defined by
\begin{eqnarray}
\label{eq:MonteCarlo:bias}
B_n \eqdef  \CPE{H_{n+1}}{\F_n} - \nabla f(\theta_n) & = \pma{ m_{n+1}^{-1} \sum_{j=1}^{m_{n+1}} \CPE{
\H{\theta_n}{\Xb{n+1}{j}}}{\F_n}  - \nabla f(\theta_n)} \nonumber  \\
& = m_{n+1}^{-1} \sum_{j=1}^{m_{n+1}}  P_{\theta_n}^j \H{\theta_n}{\Xb{n+1}{0}} - \nabla f(\theta_n)  \eqsp,
\end{eqnarray}
does not vanish. When $m_n= m$ is small, the bias can even be pretty large, and
the way the bias is mitigated in the algorithm requires substantial
mathematical developments, which are not covered by the results currently
available in the literature (see
e.g.~\cite{combettes:pesquet:2014,Rosasco:Villa:Vu:2014,combettes:pesquet:2015,rosasco:villa:vu:2015,lin:rosasco:villa:zhou:2015}).

To capture in a common unifying framework these two different situations we
assume that
\begin{assumption}
\label{hyp:MonteCarloSamples}
$H_{n+1}$ is a Monte Carlo approximation of the expectation $\nabla
f(\theta_n)$ :
\[
H_{n+1} = m_{n+1}^{-1} \sum_{j=1}^{m_{n+1}} H_{\theta_n}(\Xb{n+1}{j}) \eqsp;
\]
for all $n \geq 0$, conditionally to the past, $\{ \Xb{n+1}{j}, 1 \leq j \leq
m_{n+1} \}$ is a Markov chain started from $\Xb{n}{m_n}$ and with transition
kernel $P_{\theta_n}$ (we set $\Xb{0}{m_0} =x_\star \in \Xset$). For all
$\theta \in \Theta$, $P_\theta$ is a Markov kernel with invariant distribution
$\pi_\theta$.
\end{assumption}
For a measurable function $V: \Xset \to \coint{1,\infty}$, a signed measure
$\mu$ on the $\sigma$-field of $\Xset$, and a function $f: \Xset \to
\rset$, define
\[
\normW{f}{V} \eqdef \sup_{x \in \Xset} \frac{|f(x)|}{V(x)} \eqsp, \qquad
\normWm{\mu}{V} \eqdef \sup_{f, \normW{f}{V} \leq 1} \left|\int f \, \rmd \mu
\right| \eqsp.
 \]

 \begin{assumption} \label{hyp:UniformErgo}
 \pma{  There exist $\lambda \in \ooint{0,1}$, $b < \infty$, $p \geq 2$ and
   a measurable function $W: \Xset \to \coint{1, +\infty}$ such that
 \[
 \sup_{\theta \in \Theta}  \normW{H_\theta}{W} < \infty
 \eqsp, \qquad \sup_{\theta \in \Theta} P_\theta W^p \leq \lambda W^p + b \eqsp.
 \]
 In addition, for any $\ell \in \ocint{0,p}$, there
exist $C < \infty$ and $\rho \in \ooint{0,1}$ such that for any $x \in \Xset$,
\begin{equation}\label{eq:Vnorme:ergo:geom}
\sup_{\theta \in \Theta} \| P_\theta^n(x,\cdot) - \pi_\theta \|_{W^\ell} \leq C
\rho^n W^\ell(x) \eqsp.
\end{equation}}
\end{assumption}
\pma{Sufficient conditions for the uniform-in-$\theta$ ergodic
  behavior \eqref{eq:Vnorme:ergo:geom} are given e.g. in of
~\cite{fort:moulines:priouret:2012} Lemma 2.3., in terms of
aperiodicity, irreducibility and minorization conditions on the
kernels $\{P_\theta, \theta \in \Theta \}$.} Examples of MCMC kernels
$P_\theta$ satisfying this assumption can be found in
\cite[Proposition 12]{andrieu:moulines:2006}, \cite[Proposition
15]{saksman:vihola:2010}, \cite[Proposition
3.1.]{fort:moulines:priouret:2012}, \cite[Proposition
3.2.]{schreck:fort:moulines:2013}, \cite[Proposition
1]{allassonniere:kuhn:2013}, \cite[Proposition
3.1.]{fort:jourdain:kuhn:lelievre:stoltz:2015}. 

The proof of the results below consists in verifying the conditions of
\autoref{prop:CvgCvx:PerturbedProximalGradient} with the error term defined by
$\eta_{n+1} = m_{n+1}^{-1} \sum_{j=1}^{m_{n+1}} \H{\theta_n}{X_{n+1}^{(j)}} -
\nabla f(\theta_n)$.  If the approximation is unbiased in the sense that
$\CPE{\eta_{n+1}}{\F_n}= 0$, then $\sequence{\eta}[n][\nset]$ is a martingale
increment sequence.  In all the other cases, we decompose $\eta_{n+1}$ as the
sum of a martingale increment term and a remainder term.  When the batch size
$\sequence{m}[n][\nset]$ is increasing, the martingale increment sequence can
be set to $\eta_{n+1} - \CPE{\eta_{n+1}}{\F_n}$ and the remainder term
$\CPE{\eta_{n+1}}{\F_n}$ will be shown to be vanishingly small.  When the batch
size $\sequence{m}[n][\nset]$ is constant, then $\CPE{\eta_{n+1}}{\F_n}$ does
not vanish.  A more subtle definition of the martingale increment has to be
done, introducing the Poisson equation for Markov chain (see \autoref{prop:AS:TermeFluctuations} in \autoref{sec:proofs}).

\subsection{Monte Carlo approximation with fixed batch-size}
\label{sec:MC:fixedabatch}
We first study the case when $m_n = m$ for any $n \in \nset$.
\autoref{theo:approxsto:cvg} provides sufficient conditions for the convergence
towards the limiting set $\L $ and for a bound for $\sum_{k=1}^n a_k
F(\theta_k) - \min F$.  Consider the following assumption
\begin{assumption} \label{hyp:AS:withbias}
  \begin{enumerate}[label=(\roman*)]
  \item \label{hyp:smooth:intheta} there exists a constant $C$ such that for
    any $\theta, \theta' \in \Theta$
\[
\normW{\H{\theta}{}- \H{\theta'}{}}{W} + \sup_{x}\frac{
  \normWm{P_\theta(x,\cdot) - P_{\theta'}(x,\cdot)}{W}}{W(x)} +
\normWm{\pi_\theta - \pi_{\theta'}}{W} \leq C \, \| \theta - \theta' \| \eqsp.
\]
\item \label{hyp:AS:StabiliteCompact:item2} $\sup_{\gamma \in \ocint{0,1/L}}
  \sup_{\theta \in \Theta} \gamma^{-1} \ \norm{\Prox_{\gamma,g}(\theta) - \theta} <
  \infty$.
\item  \label{hyp:AS:stepsize:bias} $\sum_n | \gamma_{n+1} - \gamma_n | < \infty$.
  \end{enumerate}
\end{assumption}

Assumption \autoref{hyp:AS:withbias}-\ref{hyp:smooth:intheta} requires
a Lipschitz-regularity in the parameter $\theta$ of the Markov kernel
$P_\theta$ which, for MCMC algorithms, is inherited under mild
additional conditions from the Lipschitz regularity in $W$-norm of the
target distribution. Such conditions have been worked out for general
families of MCMC kernels including Hastings-Metropolis dynamics, Gibbs
samplers, and hybrid MCMC algorithm; see for example Proposition 12 in
\cite{andrieu:moulines:2006}, the proof of Theorem 3.4. in
\cite{fort:moulines:priouret:2012}, Lemmas 4.6. and 4.7. in
\cite{fort:jourdain:kuhn:lelievre:stoltz:2015} and the references
therein.  It is a classical assumption when studying Stochastic
Approximation with conditionally Markovian dynamic (see
e.g.~\cite{benveniste:etal}, \cite{andrieu:moulines:priouret:2005},
\cite{fort:moulines:vihola:schreck:2014}).

We prove in \autoref{prop:SousGradient} that when $g$ is proper,
convex, Lipschitz on $\Theta$, then
\autoref{hyp:AS:withbias}-\ref{hyp:AS:StabiliteCompact:item2} is
satisfied. \pma{In particular, if $\Theta$ is a closed convex set,
  \autoref{hyp:AS:withbias}-\ref{hyp:AS:StabiliteCompact:item2} is
  satisfied with the Lasso or fused Lasso penalty. If $\Theta$ is a compact convex
  set, then
  \autoref{hyp:AS:withbias}-\ref{hyp:AS:StabiliteCompact:item2} is
  satisfied by the elastic-net penalty.}

For a random variable $Y$, denote by $\normLq{Y}{q}= (\PE[|Y|^q])^{1/q}$.
\begin{theorem}
  \label{theo:approxsto:cvg}
  Assume $\Theta$ is bounded. Let $\{\theta_n, n \geq 0 \}$ be given by
  \autoref{sto_proxalgo} with $\gamma_n \in \ocint{0,1/L}$ for any $n \geq 0$.
  Assume \autoref{A1}--\autoref{hyp:UniformErgo}, $m_n = m \geq 1$ and, if the
  Monte Carlo approximation is biased, assume also \autoref{hyp:AS:withbias}.
\begin{enumerate}[label=(\roman*)]
\item Assume that $\sum_n \gamma_n = \infty$ and $\sum_n \gamma_n^{2} <
  \infty$. With probability one, there exists $\theta_\infty \in \L$ such that
  $\lim_{n \to \infty} \theta_n = \theta_\infty$.
\item For any $q \in \ocint{1,p/2}$ there exists a constant $C$ such that for
  any non-negative numbers $\{a_0, \cdots, a_n \}$
\begin{multline*}
\normLq{\sum_{k=1}^n a_k \left\{ F(\theta_k) - \min F \right\}}{q}
\\
\leq C \left( \frac{a_0}{\gamma_0} + \sum_{k=1}^n \left| \frac{a_k}{\gamma_k} -
    \frac{a_{k-1}}{\gamma_{k-1}} \right| + \left(\sum_{k=1}^n a_k^2
  \right)^{1/2} + \sum_{k=1}^n a_k \gamma_k + \upsilon \sum_{k=1}^n \left| a_{k} - a_{k-1} \right| \right)
\end{multline*}
and
\begin{multline*}
\sum_{k=1}^n a_k \{ \PE[ F(\theta_k) ] - \min F \} \\
\leq
C \left( \frac{a_0}{\gamma_0} + \sum_{k=1}^n \left| \frac{a_k}{\gamma_k} -
    \frac{a_{k-1}}{\gamma_{k-1}} \right| + \sum_{k=1}^n a_k \gamma_k + \upsilon \sum_{k=1}^n \left| a_{k} - a_{k-1} \right| \right)
\end{multline*}
where $\upsilon = 0$ if the Monte Carlo approximation is unbiased and $\upsilon
=1$ otherwise.
\end{enumerate}
\end{theorem}
\begin{proof}
  The proof is postponed to \autoref{sec:proof:MC:fixedabatch}.
\end{proof}
When $a_n =1$ and $\gamma_n = (n+1)^{-1/2}$, \autoref{theo:approxsto:cvg} shows
that when $n \to \infty$,
\[
\normLq{ n^{-1} \sum_{k=1}^n F(\theta_k) - \min F }{q} = O\left(
  \frac{1}{\sqrt{n}}\right).
\]
\pma{ An upper bound $O(\ln n/\sqrt{n})$ can be obtained from Theorem
  \ref{theo:approxsto:cvg} by choosing $a_n=\gamma_n = (n+1)^{-1/2}$.
}

\subsection{Monte Carlo approximation with increasing batch size}
\label{sec:MC:increasingbatch}
The key property to discuss the asymptotic behavior of the algorithm is the
following result
\begin{proposition}\label{prop:erreurLp:sumMC}
  Assume \autoref{A2}, \autoref{hyp:MonteCarloSamples} and
  \autoref{hyp:UniformErgo}. There exists a constant $C$ such that w.p. 1 for
  any $n \geq 0$,
\begin{align*}
  \| \CPE{\eta_{n+1}}{\F_n} \| \leq C \, m_{n+1}^{-1} W(X_n^{(m_n)})\eqsp,
  \qquad \PE\left[\| \eta_{n+1}\|^p\vert \F_n \right] \leq C \ m_{n+1}^{-p/2}
  \, W^p(X_n^{(m_n)}) \eqsp.
\end{align*}
\end{proposition}
\begin{proof}
  The first inequality follows from (\ref{eq:MonteCarlo:bias}) and
  (\ref{eq:Vnorme:ergo:geom}). The second one is established
  in~\cite[Proposition 12]{fort:moulines:2003}.
\end{proof}

\begin{theorem}
  \label{theo:IncreasingBatch}
  Assume $\Theta$ is bounded. Let $\{\theta_n, n \geq 0 \}$ be given by
  \autoref{sto_proxalgo} with $\gamma_n \in \ocint{0,1/L}$ for any $n \geq 0$.
  Assume \autoref{A1}--\autoref{hyp:UniformErgo}.
\begin{enumerate}[label=(\roman*)]
\item Assume $\sum_n \gamma_n = +\infty$, $\sum_n \gamma_{n+1}^2 m_{n+1}^{-1}<
  \infty$ and, if the approximation is biased, $\sum_n \gamma_{n+1}
  m_{n+1}^{-1} < \infty$.  With probability one, there exists $\theta_\infty
  \in \L $ such that $ \lim_{n \to \infty} \theta_n = \theta_\infty$.
\item For any $q \in \ocint{1,p/2}$, there exists a constant $C$ such that for
  any non-negative numbers $\{a_0, \cdots, a_n \}$
\begin{multline*}
  \normLq{\sum_{k=1}^n a_k \left\{ F(\theta_k) - \min F \right\}}{q}
  \\
  \leq C \left( \frac{a_0}{\gamma_0} + \sum_{k=1}^n \left| \frac{a_k}{\gamma_k}
      - \frac{a_{k-1}}{\gamma_{k-1}} \right| + \left(\sum_{k=1}^n a_k^2
      m_{k}^{-1} \right)^{1/2} + \sum_{k=1}^n a_k \gamma_k m_{k}^{-1}+
    \upsilon \sum_{k=1}^n a_{k} m_{k}^{-1} \right)
\end{multline*}
and
\begin{multline*}
  \sum_{k=1}^n a_k \{ \PE[ F(\theta_k) ] - \min F \} \\
  \leq C \left( \frac{a_0}{\gamma_0} + \sum_{k=1}^n \left| \frac{a_k}{\gamma_k}
      - \frac{a_{k-1}}{\gamma_{k-1}} \right| + \sum_{k=1}^n a_k \gamma_k
    m_k^{-1} + \upsilon \sum_{k=1}^n a_k m_k^{-1} \right) \eqsp,
\end{multline*}
where $\upsilon = 0$ if the Monte-Carlo approximation is unbiased and $\upsilon =1$ otherwise.
\end{enumerate}
\end{theorem}
\begin{proof}
  See \autoref{theo:proof:IncreasingBatch}.
\end{proof}
\autoref{theo:IncreasingBatch} shows that when $n \to \infty$,
\[
\normLq{ \left( \sum_{k=1}^n a_k \right)^{-1} \sum_{k=1}^n a_k F(\theta_k) -
  \min F }{q} = O\left( \frac{\ln n }{n}\right)
\]
by choosing a fixed stepsize $\gamma_n = \gamma$, a linearly increasing
batch-size $m_n \sim n$ and a uniform weight $a_n =1$.  Note that this is the
rate after $n$ iterations of the Stochastic Proximal Gradient algorithm but
$\sum_{k=1}^n m_k = O(n^2)$ Monte Carlo samples.  Therefore, the rate of
convergence expressed in terms of complexity is $O(\ln n / \sqrt{n})$.

\section{Application to network structure estimation}
\label{sec:ex2}
To illustrate the algorithm we consider the
problem of fitting discrete graphical models in a setting where the number of
nodes in the graph is large compared to the sample size. Let $\Xset$ be a
nonempty finite set, and $p\geq 1$ an integer. We consider a graphical model on
$\Xset^p$ with joint probability mass function
\begin{equation}\label{model}
f_\theta(x_1,\ldots,x_{p})=\frac{1}{Z_\theta}\exp\left\{ \sum_{k=1}^p\theta_{kk}B_0(x_k) + \sum_{1\leq j< k\leq p} \theta_{kj}B(x_k,x_j)\right\},\end{equation}
for a non-zero function $B_0:\;\Xset\to\rset$  and a symmetric non-zero function
$B:\;\Xset\times\Xset\to \rset$. The term $Z_\theta$ is
the normalizing constant of the distribution (the partition function), which cannot (in general) be computed explicitly. The real-valued symmetric matrix
$\theta$ defines the graph structure and is the parameter of interest. It has the same interpretation as the precision matrix in a multivariate Gaussian distribution.

We consider the problem of estimating $\theta$ from  $N$ realizations $\{x^{(i)}, 1 \leq i \leq N\}$ from (\ref{model}) where $x^{(i)}=(x_1^{(i)},\ldots,x^{(i)}_p)\in\Xset^p$, and where  the true value of $\theta$ is assumed sparse. This problem is relevant for instance in biology (\cite{ekebergetal13,kamisettyetal13}), and has been considered by many authors in statistics and machine learning
(\cite{banerjeeetal08,hoefling09,ravikumaretal10,guoetal10,xueetal12}).

The main difficulty in dealing with this model is the fact that the
log-partition function $\log Z_\theta$ is intractable in general.  As a result,
most of the existing works estimate $\theta$ by using the sub-optimal approach
of replacing the likelihood function by a pseudo-likelihood function. One
notable exception that tackles the log-likelihood function is
\cite{hoefling09}, using an active set strategy (to preserve sparsity), and the junction tree
algorithm for computing the partial derivatives of the log-partition function.
However, the success of this strategy depends crucially on the sparsity of the
solution\footnote{Indeed the implementation of their algorithm in the
  \textsf{BMN} package is very sensitive to the sparsity of the solution, and their solver typically fails to converge
  if the regularization parameter is not large enough to produce a sufficiently sparse solution. In our numerical experiments, we were not able to obtain a successful run from their package for $p=100$.}. We will see that Algorithm
\ref{sto_proxalgo} implemented with a MCMC approximation of the gradient gives
a simple and effective approach for computing the penalized maximum likelihood estimate
of $\theta$.

Let $\M_p$ denote the space of $p\times p$ symmetric matrices equipped with
  the (modified) Frobenius inner product
  \[\pscal{\theta}{\vartheta}\eqdef \sum_{1\leq k\leq j\leq p}\theta_{jk}\vartheta_{jk}, \mbox{ with norm }\;\; \|\theta\| \eqdef \sqrt{\pscal{\theta}{\theta}}. \]
Equipped with this norm, $\M_p$ is the same space as the Euclidean space $\rset^d$  where $d=p(p+1)/2$.  Using a $\ell^1$-penalty on $\theta$, we see that the computation of the
  penalized maximum likelihood estimate of $\theta$ is a problem of the form
  (P) with $F = -\ell +g$ where
  \[\ell(\theta) = \frac{1}{N} \sum_{i=1}^N \pscal{\theta}{\bar B(x^{(i)})} - \log Z_\theta \mbox{ and }\;\; g(\theta)=\lambda\sum_{1\leq k \leq j\leq p}|\theta_{jk}| \eqsp;
  \]
  the matrix-valued function $\bar B: \Xset^p \to \rset^{p \times p}$ is
  defined by
  \[
  \bar B_{kk}(x) = B_0(x_k) \qquad \bar B_{kj}(x) = B(x_k, x_j)\eqsp, k \neq j
  \eqsp.
  \]
  It is easy to see that in this example, Problem (P) admits at least one
  solution $\theta_\star$ that satisfies $\lambda\sum_{1\leq k \leq j\leq
    p}|\theta_{jk}| \leq p \log |\Xset|$, where $|\Xset|$ denotes the size of
  $\Xset$.  To see this, note that since $f_\theta(x)$ is a probability,
  $-\ell(\theta) =-N^{-1}\sum_{i=1}^N \log f_\theta(x^{(i)}) \geq 0$.  Hence
  $F(\theta)\geq g(\theta) \to\infty$, as $\sum_{1\leq k \leq j\leq
    p}|\theta_{jk}| \to\infty$ and since $F$ is continuous, we conclude that it
  admits at least one minimizer $\theta_\star$ that satisfies
  $F(\theta_\star)\leq F({\bf 0}) = \log Z_{{\bf 0}} = p\log |\Xset|$. As a
  result, and without any loss of generality, we consider Problem (P) with the
  penalty $g$ replaced by $g(\theta) = \lambda\sum_{1\leq k \leq j\leq
    p}|\theta_{jk}| +\1(\theta)$, where $\1(\theta) = 0$ if $\max_{ij}
  |\theta_{ij}| \leq (p/\lambda)\log |\Xset|$, and $\1(\theta)=+\infty$
  otherwise. Hence in this problem, the domain of $g$ is $\Theta =
  \{\theta\in\M_p:\;\max_{ij} |\theta_{ij}| \leq (p/\lambda)\log |\Xset|\}$.
  
  Upon noting that (\ref{model}) is a canonical exponential model,
  \cite[Section 4.4.2]{shao:2003} shows that $\theta \mapsto -\ell(\theta)$ is
  convex and
  \begin{eqnarray}\label{grad:gibbs}
    \nabla\ell(\theta)=\frac{1}{N}\sum_{i=1}^N \bar B(x^{(i)}) -\int_{\Xset^p} \bar B(z) f_\theta(z) \mu(\rmd z) \eqsp,
  \end{eqnarray}
  where $\mu$ is the counting measure on $\Xset^p$.
  In addition, (see \autoref{app:example:network})
  \begin{equation}\label{eq:propExample}
    \|\nabla\ell(\theta)-\nabla\ell(\vartheta)\| \leq p \left((p-1) \textsf{osc}^2(B) + \textsf{osc}^2(B_0)   \right)\|\theta-\vartheta\|,\end{equation}
  where for a function $\tilde B: \Xset \times \Xset \to \rset$, $\textsf{osc}(\tilde B)=\sup_{x,y,u,v\in\Xset}|\tilde B(x,y)- \tilde B(u,v)|$.

 Therefore, in this example, the assumption \autoref{A1} and
  \autoref{A1:compl} are satisfied.

The representation of the gradient in (\ref{grad:gibbs}) shows that H\ref{A2} holds, with $\pi_\theta(\rmd z) = f_\theta(z) \mu(\rmd z)$, and $H_\theta(z) = N^{-1}\sum_{i=1}^N \bar B(x^{(i)}) -\bar B(z)$. Direct simulation from the distribution $f_\theta$ is rarely feasible, so we turn to MCMC.  These Markov kernels are easy to construct, and can be constructed in many ways. For instance if the set $\Xset$ is not too large, then a Gibbs sampler (see e.g. \cite{casella:robert:2005}) that samples from the full conditional distributions of $f_\theta$ can be easily implemented. In the case of the Gibbs sampler, since $\Xset^p$ is a finite set,  $\Theta$ is compact, $f_\theta(x)>0$ for all $(x,\theta)\in \Xset^p\times \Theta$, and, $\theta\mapsto f_\theta(x)$ is continuously differentiable, the assumptions  H\ref{hyp:MonteCarloSamples}, H\ref{hyp:UniformErgo} and H\ref{hyp:AS:withbias}(i)-(ii) automatically hold with $W\equiv 1$. We should point out that the Gibbs sampler is a generic algorithm that in some cases is known to mix poorly. Whenever possible we recommend the use of specialized problem-specific MCMC algorithms with better mixing properties. 

\paragraph{Illustrative example}
We consider the particular case where $\Xset=\{1,\ldots,M\}$, \pma{$B_0(x)=0$}, and
$B(x,y)=\1_{\{x=y\}}$, which corresponds to the well known Potts model. We
report in this section some simulation results showing the performances of the
stochastic proximal gradient algorithm. We use $M=20$, $B_0(x)=x$, $N=250$ and
for $p\in\{50,100,200\}$. We generate the ``true" matrix $\theta_\true$ such
that it has on average $p$ non-zero elements below the diagonal which are
simulated from a uniform distribution on $\ooint{-4,-1}\cup \ooint{1,4}$. All
the diagonal elements are set to $0$.

By trial-and-error we set the regularization parameter to
$\lambda=2.5\sqrt{\log(p)/n}$ for all the simulations.  We implement
\autoref{sto_proxalgo}, drawing samples from a Gibbs sampler to approximate the
gradient.  We compare the following two versions of \autoref{sto_proxalgo}:
\begin{enumerate}
\item \textsf{Solver 1}: A version with a fixed Monte Carlo batch size $m_n=500$, and decreasing
  step size $\gamma_n = \frac{25}{p}\frac{1}{n^{0.7}}$.
\item \textsf{Solver 2}: A version with increasing Monte Carlo batch size $m_n=500 + n^{1.2}$,
  and fixed step size $\gamma_n = \frac{25}{p}\frac{1}{\sqrt{50}}$.
\end{enumerate}
We run \textsf{Solver 2} for $\textsf{Niter}=5p$ iterations, where $p\in\{50,100,200\}$ is as above. And we set the number of iterations of \textsf{Solver} 1 so that both solvers draw approximately the same number of Monte Carlo samples. \pma{For stability in the results, we repeat the solvers $30$ times and average the sample paths. We evaluate the convergence of each solver by computing the relative error $\|\theta_n-\theta_\infty\| / \|\theta_\infty\|$, along the iterations, where $\theta_\infty$ denotes the value returned by the solver on its last iteration. Note that we compare the optimizer output to $\theta_\infty$, not $\theta_\true$. Ideally, we would like to compare the iterates to the solution of the optimization problem. However in the present setting a solution is not available in closed form (and there could be more than one solution). Furthermore, whether the solution of the optimization problem  approaches $\theta_\star$ is a complicated statistical problem\footnote{this depends heavily on $n$, $p$, the actual true matrix $\theta_\true$, and depends also heavily the choice of the regularization parameter $\lambda$} that is beyond the scope of this work.  The relative errors are presented on \autoref{fig:errors} and suggest that, when measured as function of resource used, \textsf{Solver 1} and \textsf{Solver 2} have roughly the same convergence rate. }

We also compute the statistic $\textsf{F}_n \eqdef
\frac{2\textsf{Sen}_n\textsf{Prec}_n}{\textsf{Sen}_n +\textsf{Prec}_n}$ which
measures the recovery of the sparsity structure of $\theta_{\infty}$ along the
iteration. In this definition $\textsf{Sen}_n$ is the sensitivity, and
$\textsf{Prec}_n$ is the precision defined as
\begin{multline*}
\textsf{Sen}_n=\frac{\sum_{j< i}\1_{\{|\theta_{n,ij}|>0\}}\1_{\{|\theta_{\infty,ij}|>0\}}}{\sum_{j< i}\1_{\{|\theta_{\infty,ij}|>0\}}},\;\;\mbox{ and }\;\;\textsf{Prec}_n=\frac{\sum_{j< i}\1_{\{|\theta_{n,ij}|>0\}}\1_{\{|\theta_{\infty,ij}|>0\}}}{\sum_{j< i}\1_{\{|\theta_{\infty,ij}|>0\}}}.\end{multline*}
The values of $\textsf{F}_n$ are presented on \autoref{fig:recovery} as function of computing time. It shows that for both solvers, the sparsity structure of $\theta_n$ converges very quickly towards that of $\theta_\infty$. \pma{We note also that \autoref{fig:recovery} seems to suggest that \textsf{Solver 2} tends to produce solutions with slightly more stable sparsity structure than \textsf{Solver 1} (less variance on the red curves). Whether such subtle differences exist between the two algorithms (a diminishing step-size and fixed Monte Carlo size versus a fixed step-size and increasing Monte Carlo size) is an interest question. Our analysis does not deal with the sparsity structure of the solutions, hence cannot offer any explanation.}

\begin{figure}
\includegraphics[width=0.9\textwidth]{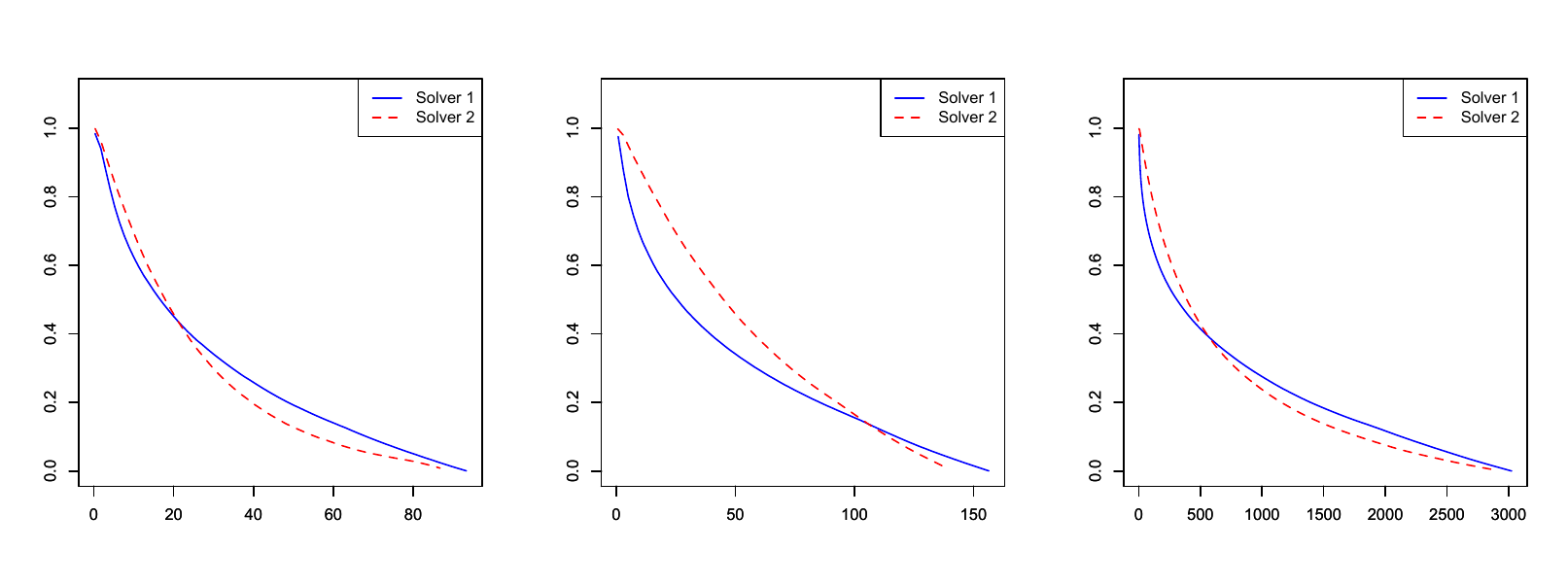}
\caption{Relative errors plotted as function of computing time for \textsf{Solver 1} and \textsf{Solver 2}. }
\label{fig:errors}
\end{figure}

\begin{figure}
\includegraphics[width=0.9\textwidth]{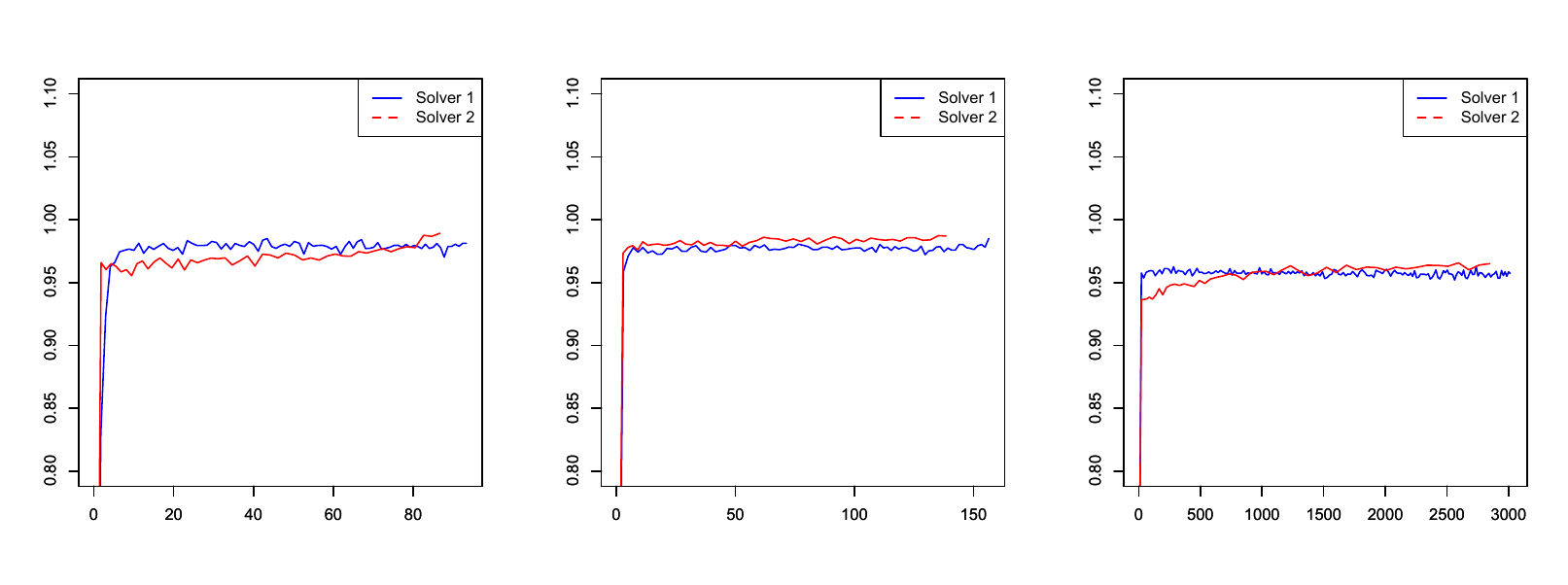}
\caption{Statistic $\textsf{F}_n$ plotted as function of computing time for \textsf{Solver 1} and \textsf{Solver 2}.}
\label{fig:recovery}
\end{figure}

\section{A non convex example: High-dimensional logistic regression with random effects}
\label{sec:ex1}
 We numerically investigate the extension of our results to a
 situation where the assumptions \autoref{A1:compl} and \autoref{A2}
 hold but \autoref{A1} is not in general satisfied \pma{and the domain
   $\Theta$ is not bounded}. The numerical study below shows that the
 conclusions reached in \autoref{sec:stochastic-proximal-gradient} and
 \autoref{sec:MonteCarloProxGdt} provide useful information to tune
 the design parameters of the algorithms.

 \subsection{The model}  We model binary responses $\{ Y_{i} \}_{i=1}^N \in \{0,1\}$  as $N$ conditionally independent realizations of a
  random effect logistic regression model,
  \begin{equation}\label{randeffect_model}
    Y_{i}\vert \bU \stackrel{ind.}{\sim} \textsf{Ber}\left(s(x_i'\beta +\sigma z_i'\bU)\right),\;\;1\leq i\leq N \eqsp,
  \end{equation}
  where $x_i\in\rset^{p}$ is the vector of covariates, $z_i\in\rset^{q}$ are (known) loading vector,
  $\textsf{Ber}(\alpha)$ denotes the Bernoulli distribution with parameter
  $\alpha \in \ccint{0,1}$, $s(x)=\rme^{x}/(1+\rme^{x})$ is the cumulative
  distribution function of the standard logistic distribution.
  The random effect $\bU$ is assumed to be standard Gaussian $\bU \sim \gauss_q(0,I)$.
  
  The log-likelihood of the observations at $\theta = (\beta, \sigma) \in
  \rset^p \times \ooint{0, \infty}$ is given by
  \begin{equation}
    \label{eq:fun_ex:LogLikeli}
    \ell(\theta)=\log \int \prod_{i=1}^N s(x_i'\beta + \sigma z_i'\bu)^{Y_{i}}\left(1-s(x_i'\beta + \sigma z_i'\bu)\right)^{1-Y_{i}}\phi(\bu) \rmd \bu \eqsp, \\
  \end{equation}
  where $\phi$ is the density of a $\rset^q$-valued standard Gaussian random
  vector.  The number of covariates $p$ is possibly larger than $N$, but only a
  very small number of these covariates are relevant which suggests to use the
  elastic-net penalty
  \begin{equation}
    \label{eq:elasticnet}
    \lambda \left(\frac{1-\alpha}{2} \| \beta\|_2^2 + \alpha \|\beta \|_1 \right)  \eqsp,
  \end{equation}
  where $\lambda >0$ is the regularization parameter, $\| \beta \|_r =
  (\sum_{i=1}^p |\beta_i|^r)^{1/r}$ and $\alpha \in \ccint{0,1}$ controls the
  trade-off between the $\ell^1$ and the $\ell^2$ penalties.  In this example,
  \begin{equation}
    \label{eq:ex_fun:penalty}
    g(\theta) = \lambda\left(\frac{1-\alpha}{2}\|\beta\|_2^2
      +\alpha\|\beta\|_1\right) +   \1_{\ooint{0, +\infty}}(\sigma) \eqsp,
  \end{equation}
  where $\1_A(x) = +\infty$ is $x \notin A$ and $0$ otherwise.  Define the
  conditional log-likelihood of $\bY=(Y_1,\dots,Y_N)$ given $\bU$ (the
  dependence upon $\bY$ is omitted) by
  \[
  \ell_c(\theta \vert\bu) = \sum_{i=1}^N \left\{ Y_i \left( x_i' \beta + \sigma
      z_i'\bu \right) - \ln \left( 1+\exp\left( x_i'\beta + \sigma
        z_i'\bu\right)\right) \right\} \ \eqsp,
  \]
  and the conditional distribution of the random effect $\bU$ given the observations
  $\bY$ and the parameter $\theta$
  \begin{equation}
    \label{eq:posterior-density-random-effect}
    \pi_{\theta}(\bu)= \exp\left(\ell_c(\theta \vert \bu) - \ell(\theta) \right) \phi(\bu) \eqsp.
  \end{equation}
  The Fisher identity implies that the gradient of the
  log-likelihood~(\ref{eq:fun_ex:LogLikeli}) is given by
  \begin{align*}
    \nabla \ell(\theta)    &= \int \nabla_\theta \ell_c(\theta \vert \bu) \ \pi_\theta(\bu) \ \rmd \bu = \int\left\{\sum_{i=1}^N (Y_{i}-s(x_i'\beta+ \sigma z_i'\bu))
      \begin{bmatrix}
        x_i \\
        z_i' \bu
      \end{bmatrix}
    \right\}\pi_{\theta}(\bu) \ \rmd \bu \eqsp.
  \end{align*}
  The Hessian of the
  log-likelihood $\ell$ is given by (see e.g.\cite[Chapter 3]{Mclachlan:2008})
  \[
  \nabla^2 \ell(\theta) = \PE_{\pi_\theta}\left[ \nabla^2_\theta \ell_c(\theta
    \vert \bU) \right] + \mathrm{Cov}_{\pi_\theta}\left( \nabla_\theta
    \ell_c(\theta \vert \bU) \right)
  \]
  where $\PE_{\pi_\theta}$ and $\mathrm{Cov}_{\pi_\theta}$ denotes  the
  expectation and the covariance \wrt\ the distribution $\pi_\theta$, respectively.
  Since
  \[
  \nabla_\theta^2 \ell_c(\theta \vert \bu)= - \sum_{i=1}^N s(x_i' \beta + \sigma
  z_i' \bu) \, \left(1-s(x_i' \beta + \sigma z_i' \bu) \right) \begin{bmatrix}
    x_i \\
    z_i' \bu
  \end{bmatrix} \begin{bmatrix}
    x_i \\
    z_i' \bu
  \end{bmatrix}' \eqsp,
  \]
  and $\sup_{\theta \in \Theta} \int \|\bu\|^2 \pi_\theta(\bu) \, \rmd \bu <
  \infty$ (see \autoref{app:ex_fun}), $\nabla^2 \ell(\theta)$ is bounded on
  $\Theta$. Hence, $\nabla \ell(\theta)$ satisfies the Lipschitz condition showing that~\autoref{A1}
  is satisfied.

 \subsection{Numerical application}
 The assumption \pma{\autoref{A2}} is satisfied with $\pi_\theta$ given by
 (\ref{eq:posterior-density-random-effect}) and
\begin{equation}\label{eq:example2:defi:Htheta}
  H_\theta(\bu) =-\sum_{i=1}^N (Y_{i}-F(x_i'\beta+ \sigma z_i'\bu))
  \begin{bmatrix}
    x_i \\
    z_i' \bu
  \end{bmatrix} \eqsp.
\end{equation}
The distribution $\pi_\theta$ is sampled using the MCMC sampler proposed in
\cite{polson:scott:windle:2012} based on data-augmentation.  We write $-\nabla\ell(\theta) = \int_{\rset^q \times
  \rset^N} H_\theta(\bu) \tpi_\theta(\bu,\bw) \ \rmd \bu \rmd \bw$ where
$\tpi_\theta(\bu, \bw)$ is defined for $\bu \in \rset^q$ and $\bw = (w_1,
\cdots, w_N) \in \rset^N$ by
\[
\tpi_\theta(\bu, \bw) = \left(\prod_{i=1}^N \bpi_{\pg}\left(w_i; x_i'\beta+
    \sigma z_i'\bu \right)\right) \pi_\theta(\bu) \eqsp;
\]
in this expression, $\bpi_\pg(\cdot;c)$ is the density of the Polya-Gamma
distribution on the positive real line with parameter $c$ given by
\[
\bpi_\pg(w;c) = \cosh(c/2) \, \exp\left( -w c^2/2 \right) \, \rho(w) \un_{\rset^+}(w)  \eqsp,
\]
where $\rho(w) \propto \sum_{k \geq 0} (-1)^k (2k+1) \exp(-(2k+1)^2/(8w))
w^{-3/2}$ (see \cite[Section~3.1]{biane:pitman:yor:2001}).  Thus, we have
\[
\tpi_\theta(\bu,\bw) = C_\theta \phi(\bu) \prod_{i=1}^N  \exp\left( \sigma (Y_i-1/2)  z_i' \bu  -w_i (x_i'\beta+ \sigma z_i'\bu)^2/2 \right) \, \rho(w_i) \un_{\rset^+}(w_i)  \eqsp,
\]
where $ \ln C_\theta = -N \ln 2 -\ell(\theta) +\sum_{i=1}^N (Y_i-1/2) x'_i
\beta$.  This target distribution can be sampled using a Gibbs algorithm: given
the current value $(\bu^t, \bw^t)$ of the chain, the next point is obtained by
sampling $\bu^{t+1}$ under the conditional distribution of $\bu$ given $\bw^t$,
and $\bw^{t+1}$ under the conditional distribution of $\bw$ given $\bu^{t+1}$.
In the present case, these conditional distributions are given respectively by
\[
  \tpi_\theta(\bu \vert \bw)  \equiv \gauss_q\left( \mu_\theta(\bw) ; \Gamma_\theta(\bw) \right)  \qquad
  \tpi_\theta(\bw \vert \bu)  = \prod_{i=1}^N \bpi_\pg(w_i; |x_i'\beta +
  \sigma z_i'\bu|)
\]
with
\begin{equation}\label{eq:example2:defi:MuVar}
  \Gamma_\theta(\bw) = \left(I + \sigma^2 \sum_{i=1}^N w_i z_i z_i' \right)^{-1} \eqsp, \qquad \mu_\theta(\bw) = \sigma \Gamma_\theta(\bw)  \sum_{i=1}^N \left( (Y_i-1/2) - w_i x_i' \beta \right) z_i \eqsp.
\end{equation}
Exact samples of these conditional distributions can be obtained (see
\cite[Algorithm 1]{polson:scott:windle:2012} for sampling under a Polya-Gamma
distribution).
It has been shown by \cite{choi:hobert:2013} that the Polya-Gamma Gibbs sampler is uniformly ergodic. 
Hence \autoref{hyp:UniformErgo} is satisfied with $W \equiv 1$. 
Checking \autoref{hyp:AS:withbias} is also straightforward.

We test the algorithms with $N=500$, $p=1,000$ and $q=5$.  We generate the
$N\times p$ covariates matrix $X$ columnwise, by sampling a stationary
$\rset^N$-valued autoregressive model with parameter $\rho =0.8$ and Gaussian
noise $\sqrt{1-\rho^2} \, \N_N(0, I)$.  We generate the vector of regressors
$\beta_\true$ from the uniform distribution on $[1,5]$ and randomly set $98 \%$
of the coefficients to zero. The variance of the random effect is set to $\sigma^2=0.1$.  We
consider a repeated measurement setting so that $z_i = e_{\lceil i q /N
  \rceil}$ where $\{e_j, j \leq q \}$ is the canonical basis of $\rset^q$ and
$\lceil \cdot \rceil$ denotes the upper integer part.  With such a simple
expression for the random effect, we will be able to approximate the value
$F(\theta)$ in order to illustrate the theoretical results obtained in this
paper.  We use the Lasso penalty ($\alpha=1$ in (\ref{eq:elasticnet})) with
$\lambda=30$.

We first illustrate the ability of Monte Carlo Proximal Gradient algorithms to
find a minimizer of $F$.  We compare the Monte Carlo proximal gradient algorithm
\begin{enumerate}[label=(\roman*)]
\item with fixed batch size: $\gamma_n = 0.01/\sqrt{n}$ and $m_n = 275$
  (Algo~1); $\gamma_n = 0.5/n$ and $m_n = 275$ (Algo~2).
\item with increasing batch size: $\gamma_n = \gamma =0.005$, $m_n = 200+n$
  (Algo~3); $\gamma_n = \gamma =0.001$, $m_n = 200+n$ (Algo~4); and $\gamma_n =
  0.05/\sqrt{n}$ and $m_n = 270+\lceil\sqrt{n}\rceil$ (Algo~5).
  %% \item (Algo~3) the Monte Carlo averaged proximal gradient algorithm with
  %%   $\gamma_n = \gamma = 0.005$, $m_n = 200+n$ and three different strategies for
  %%   the averaging sequence: $a_n=1/n^{0.1}, a_n =1$ and $a_n = \sqrt{n}$.
  %% \item (Algo~4) the Monte Carlo averaged proximal gradient algorithm with
  %%   $\gamma_n =0.05/\sqrt{n}$, $m_n = 267+\lceil \sqrt{n} \rceil$ and three
  %%   different strategies for the averaging sequence: $a_n =n^{-0.49}, a_n =1,
  %%   a_n=\sqrt{n}$.  \setcounter{saveenum}{\value{enumi}}
\end{enumerate}
Each algorithm is run for $150$ iterations. The batch sizes $\{m_n,n \geq 0\}$
are chosen so that after $150$ iterations, each algorithm used approximately
the same number of Monte Carlo samples.  We denote by $\beta_\infty$ the value
obtained at iteration $150$. A path of the relative error $\|\beta_n -
\beta_\infty\|/\|\beta_\infty \|$ is displayed on
\autoref{fig:RelError:Beta}[right] for each algorithm; a path of the
sensitivity $\mathsf{Sen}_n$ and of the precision $\mathsf{Prec}_n$ (see
\autoref{sec:ex2} for the definition) are displayed on
\autoref{fig:SensitivityPrecision:Beta}.  All these sequences are plotted
versus the total number of Monte Carlo samples up to iteration $n$.  These
plots show that with a fixed batch-size (Algo~1 or Algo~2), the best
convergence is obtained with a step size decreasing as $O(1/\sqrt{n})$; and for
an increasing batch size (Algo~3 to Algo~5), it is better to choose a fixed
step size. These findings are consistent with the results in
\autoref{sec:MonteCarloProxGdt}.  On \autoref{fig:RelError:Beta}[left], we
report on the bottom row the indices $j$ such that $\beta_{\true,j}$ is non
null and on the rows above, the indices $j$ such that $\beta_{\infty,j}$ given
by Algo~1 to Algo~5 is non null.

\begin{figure}  \includegraphics[width=6cm]{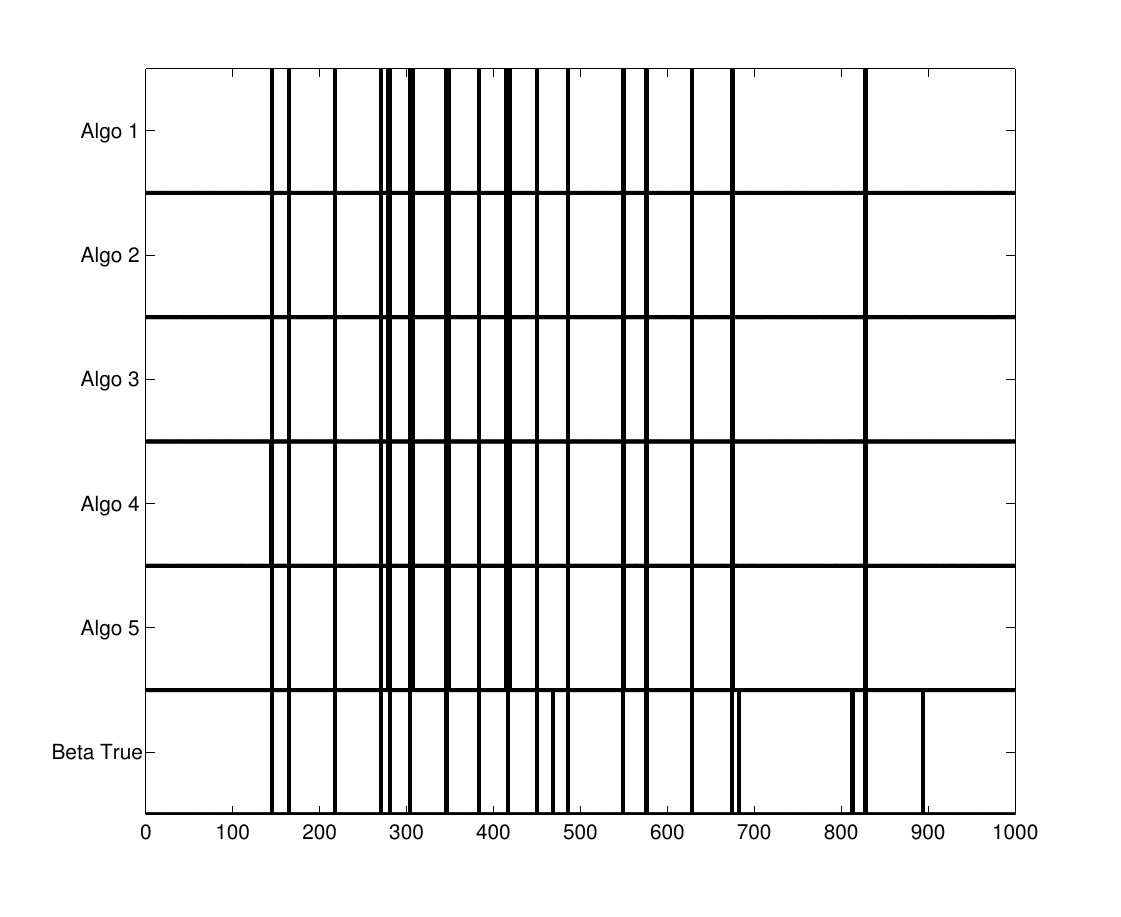} \includegraphics[width=6cm]{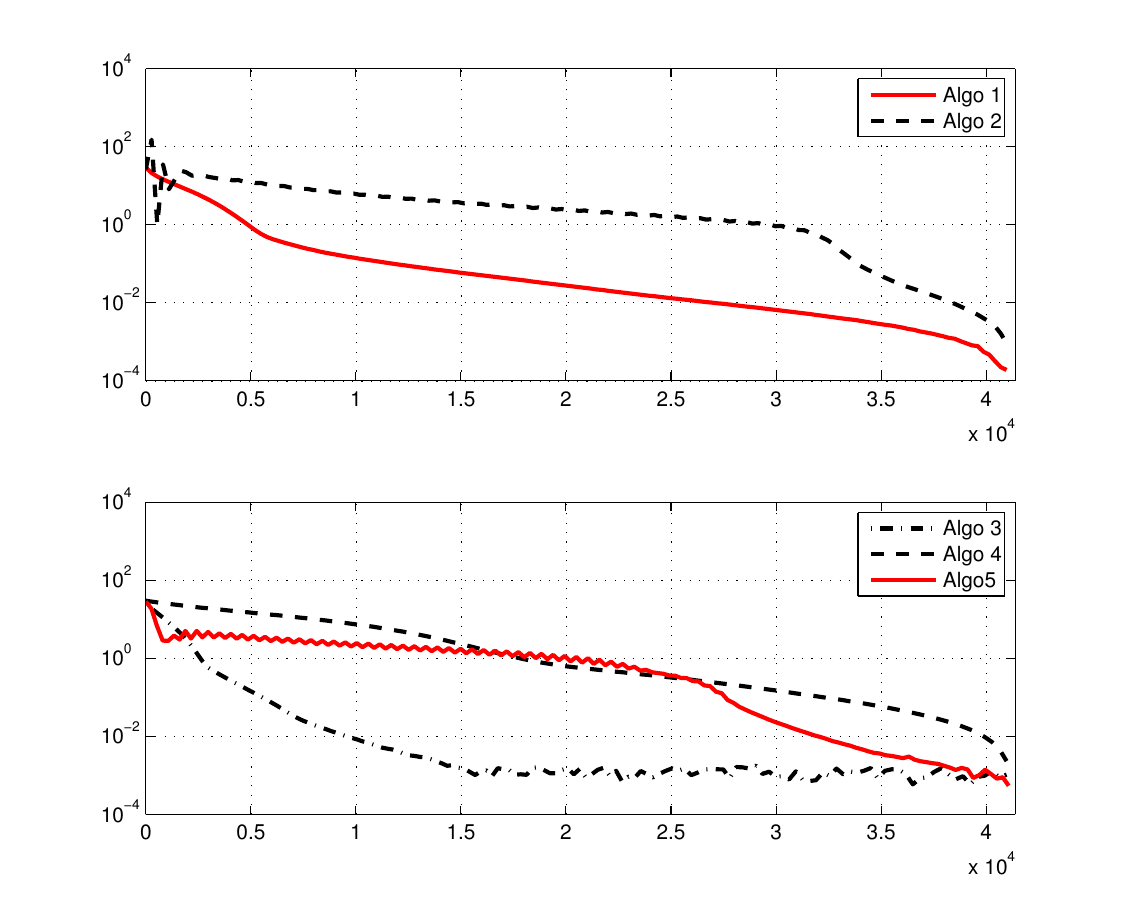}
 \caption{[left] The support of the sparse vector $\beta_\infty$
    obtained by Algo~1 to Algo~5; for comparison, the support of $\beta_\true$ is on the bottom row.
    [right] Relative error along one path of each algorithm as a function of
    the total number of Monte Carlo samples.}
  \label{fig:RelError:Beta}
\end{figure}

\begin{figure}
  \includegraphics[width=6cm]{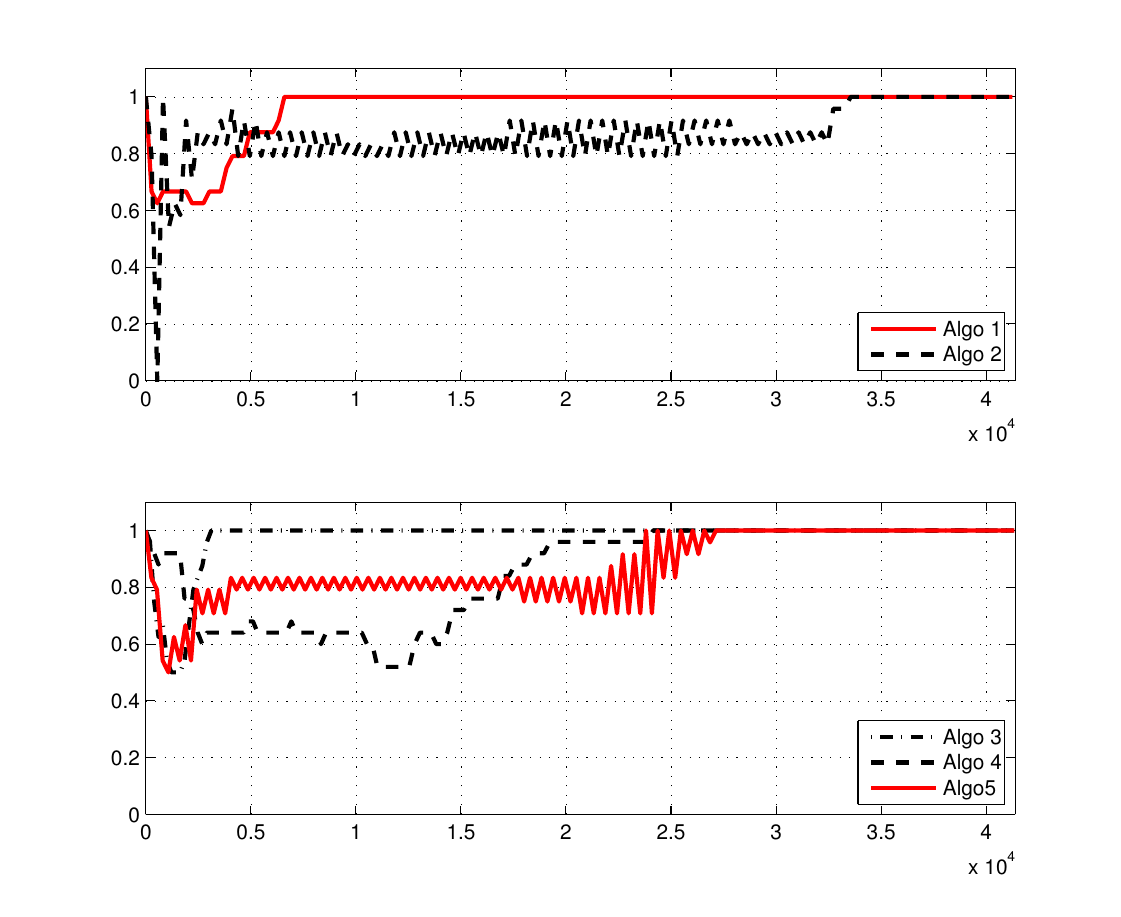} \includegraphics[width=6cm]{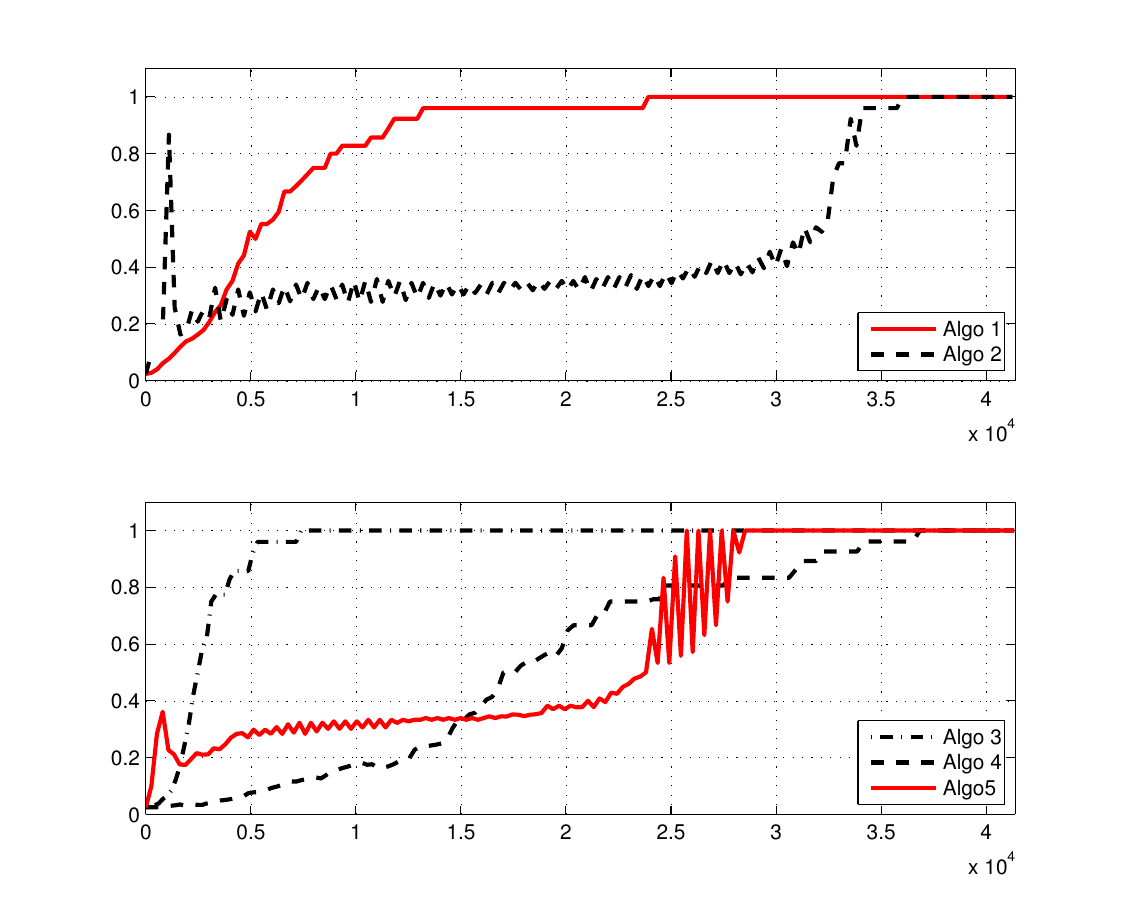}
  \caption{ The sensitivity $\mathsf{Sen}_n$ [left]  and the precision $\mathsf{Prec}_n$ [right] along a path, versus the total
    number of Monte Carlo samples up to time $n$}
  \label{fig:SensitivityPrecision:Beta}
\end{figure}
We now study the convergence of $\{F(\theta_n), n \in \nset \}$ where
$\theta_n$ is obtained by one of the algorithms described above.  We repeat
$50$ independent runs for each algorithm and estimate $\PE\left[F(\theta_n)
\right]$ by the empirical mean over these runs. On \autoref{fig:CvgF}[left], $n
\mapsto F(\theta_n)$ is displayed for several runs of Algo~1 and Algo~3.  The
figure shows that all the paths have the same limiting value, which is
approximately $F_\star=311$; we observed the same behavior on the $50$ runs of
each algorithm. On \autoref{fig:CvgF}[right], we report the Monte Carlo
estimation of $\PE[F(\theta_n)]$ versus the total number of Monte Carlo samples
used up to iteration $n$ for the best strategies in the fixed batch size case
(Algo~1) and in the increasing batch size case (Algo~3 and Algo~4).
\begin{figure}  \includegraphics[width=6cm]{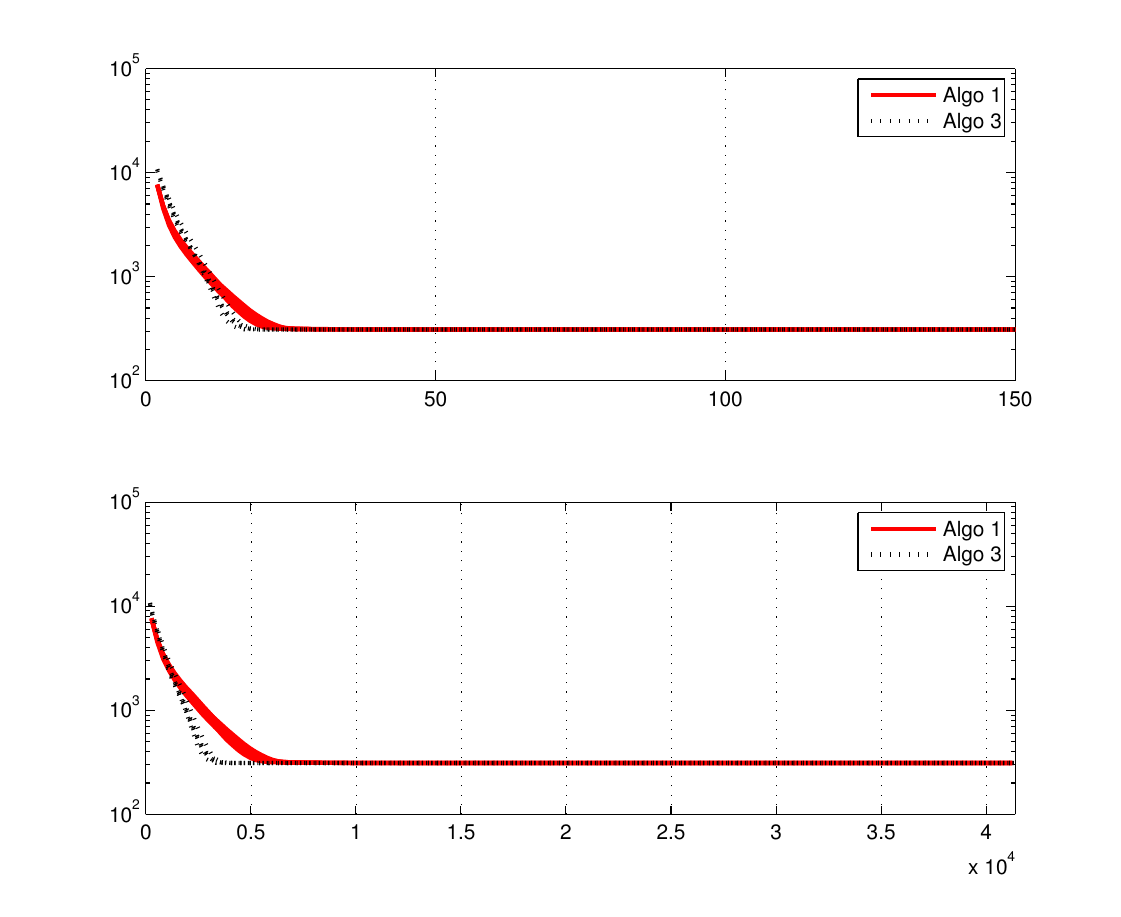} \includegraphics[width=6cm]{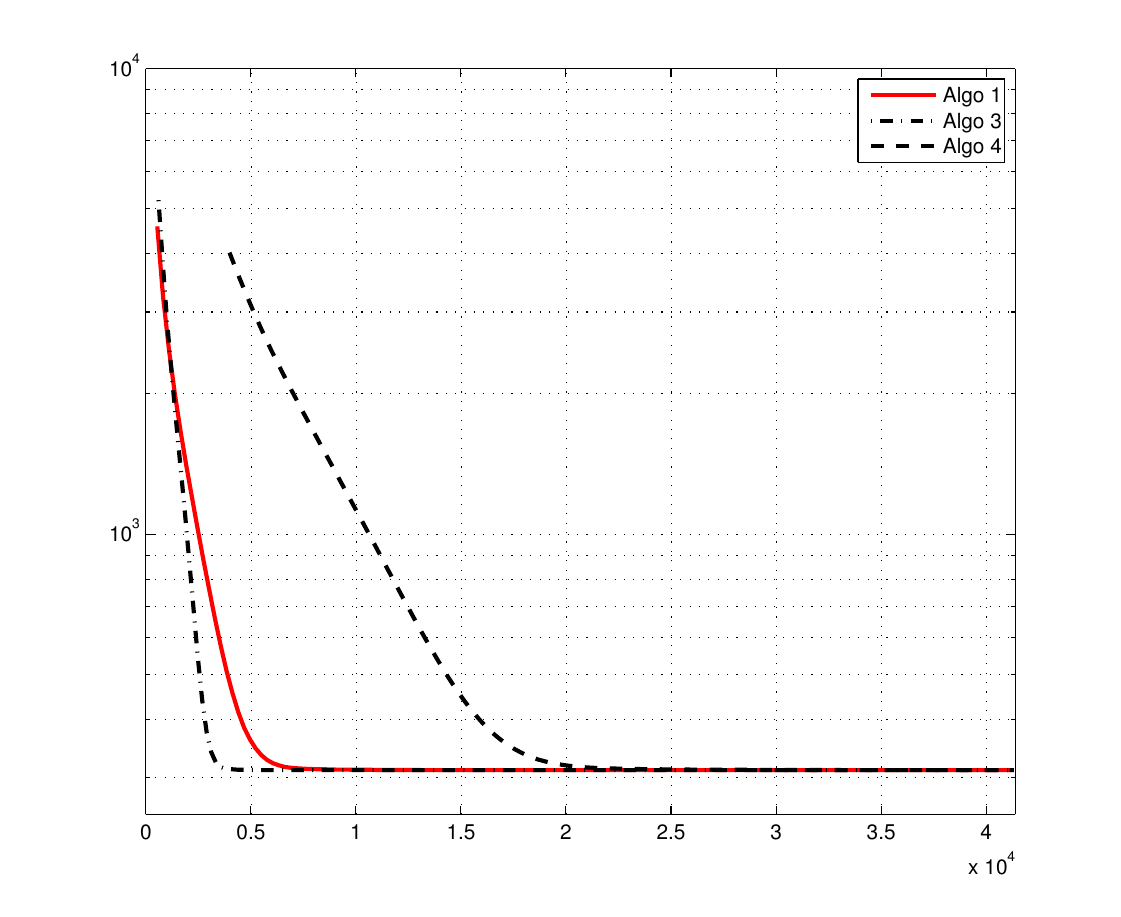}
   \caption{[left] $n \mapsto F(\theta_n)$ for several independent
     runs. [right] $\PE\left[F(\theta_n)\right]$ versus the total number of
     Monte Carlo samples up to iteration $n$}
  \label{fig:CvgF}
\end{figure}

\section{Proofs}
\label{sec:proofs}
\subsection{Preliminary lemmas}
\begin{lemma}\label{lem:prox}  Assume that $g$ is lower semi-continuous and convex.
  For $\theta,\theta'\in\Theta$ and $\gamma>0$
  \begin{equation}
    \label{eq:optcond}
    g\Big(\Prox_{\gamma,g}(\theta) \Big)-g(\theta')\leq - \frac{1}{\gamma}\pscal{\Prox_{\gamma,g}(\theta) - \theta'}{\Prox_{\gamma, g}(\theta)-\theta} \eqsp.
  \end{equation}
  For any $\gamma>0$ and for any $\theta,\theta'\in\Theta$,
  \begin{equation}
    \label{eq:Proximal-Liphscitz}
    \|\Prox_{\gamma,g}(\theta) - \Prox_{\gamma,g}(\theta') \|^2 + \| \big(\Prox_{\gamma,g}(\theta) - \theta  \big)- \big( \Prox_{\gamma,g}(\theta') - \theta' \big)\|^2 \leq \|\theta - \theta' \|^2 \eqsp.
  \end{equation}
\end{lemma}
\begin{proof}
  See \cite[Propositions~4.2., 12.26 and ~12.27]{bauschke:combettes:2011}.
\end{proof}

\begin{lemma}\label{lem:key}
  Assume \autoref{A1}  and let $\gamma \in \ocint{0,1/L}$. Then for all $\theta,\theta'\in\Theta$,
  \begin{equation}\label{eq:key2}
    - 2 \gamma \Big( F(\Prox_{\gamma,g}(\theta))-F(\theta') \Big)
    \geq \|\Prox_{\gamma,g}(\theta)-\theta'\|^2 + 2 \pscal{\Prox_{\gamma,g}(\theta)-\theta'}{\theta'-\gamma\nabla f(\theta')-\theta} \eqsp.
  \end{equation}
  If in addition $f$ is convex, then for all $\theta,\theta',\xi\in\Theta$,
  \begin{multline}
    \label{eq:key1}
    -2 \gamma \Big( F\big(\Prox_{\gamma,g}(\theta) \big) - F(\theta') \Big) \geq \|\Prox_{\gamma,g}(\theta) -
    \theta'\|^2 \\
    + 2 \pscal{\Prox_{\gamma,g}(\theta)-\theta'}{\xi - \gamma\nabla f(\xi)-\theta}
    - \|\theta'-\xi\|^2 \eqsp.
  \end{multline}
\end{lemma}
\begin{proof}
  Since $\nabla f$ is Lipschitz, the descent lemma implies that for any
  $\gamma^{-1} \geq L$
  \begin{equation}
    \label{eq:step:key2}
    f(p) - f(\theta') \leq \pscal{\nabla f(\theta')}{p-\theta'} + \frac{1}{2 \gamma} \|p - \theta'\|^2 \eqsp.
  \end{equation}
  This inequality applied with $p = \Prox_{\gamma,g}(\theta)$ combined with
  (\ref{eq:optcond}) yields (\ref{eq:key2}).  When $f$ is convex, $
  f(\xi)+\pscal{\nabla f(\xi)}{\theta'-\xi} - f(\theta') \leq 0$ which,
  combined again with (\ref{eq:optcond}) and (\ref{eq:step:key2}) applied with
  $(p,\theta') \leftarrow (\Prox_{\gamma,g}(\theta), \xi)$ yields the result.
\end{proof}

\begin{lemma} \label{lem:Lipschitz:GradientProximalMap}
  Assume \autoref{A1}. Then for any $\gamma >0$, $\theta,\theta'\in\Theta$,
  \begin{align}
    \label{eq:majoration-utile}
    &\|\theta - \gamma \nabla f(\theta) - \theta' + \gamma \nabla f(\theta') \| \leq (1+\gamma L)\| \theta - \theta' \| \eqsp, \\
    \label{eq:Tgamma-contraction}
    &\|T_\gamma(\theta) - T_\gamma(\theta') \| \leq(1+\gamma L) \| \theta - \theta' \| \eqsp.
  \end{align}
  If in addition $f$ is convex then for any $\gamma \in \ocint{0,2/L}$,
  \begin{align}
    \label{eq:majoration-utile:conv}
    &\|\theta - \gamma \nabla f(\theta) - \theta' + \gamma \nabla f(\theta') \| \leq \| \theta - \theta' \| \eqsp, \\
    \label{eq:Tgamma-contraction:conv}
    &\|T_\gamma(\theta) - T_\gamma(\theta') \| \leq \| \theta - \theta' \| \eqsp.
  \end{align}
\end{lemma}
\begin{proof}
  (\ref{eq:Tgamma-contraction}) and (\ref{eq:Tgamma-contraction:conv}) follows
  from (\ref{eq:majoration-utile}) and (\ref{eq:majoration-utile:conv})
  respectively by the Lipschitz property of the proximal map $\Prox_{\gamma,
    g}$ (see \autoref{lem:prox}). (\ref{eq:majoration-utile}) follows directly
  from the Lipschitz property of $f$. It remains to prove
  (\ref{eq:majoration-utile:conv}). Since $f$ is a convex function with
  Lipschitz-continuous gradients, \cite[Theorem 2.1.5]{nesterov:2004} shows
  that, for all $\theta, \theta' \in \Theta$, $ L \, \pscal{\nabla f(\theta) -
    \nabla f(\theta')}{\theta - \theta'} \geq \| \nabla f(\theta) - \nabla
  f(\theta') \|^2$.  The result follows.
\end{proof}

\begin{lemma}\label{lem:techlem}
  Assume \autoref{A1}.  Set $S_\gamma(\theta) \eqdef \Prox_{\gamma,g}(\theta -
  \gamma H)$ and $\eta\eqdef H -\nabla f(\theta)$.  For any $\theta \in \Theta$
  and $\gamma>0$,
  \begin{equation}
    \label{eq:util1}
    \| T_\gamma(\theta) - S_\gamma(\theta) \|  \leq \gamma \| \eta \| \eqsp.
  \end{equation}
\end{lemma}
\begin{proof}
  We have $\| T_\gamma(\theta) - S_\gamma(\theta) \| = \|
  \Prox_{\gamma,g}(\theta - \gamma \nabla f(\theta)) - \Prox_{\gamma,g}(\theta
  - \gamma H) \|$ and \eqref{eq:util1} follows from \autoref{lem:prox}.
\end{proof}

\subsection{Proof of \autoref{sec:stochastic-proximal-gradient}}
\subsubsection{Proof of \autoref{lemma:RobbinsSiegmund:deterministe}}
\label{sec:proof:lemma:RobbinsSiegmund:deterministe}
Set $w_n =v _n + \sum_{k \geq n+1} \xi_k +M$ with $M \eqdef - \inf_n \sum_{k
  \geq n} \xi_k$ so that $\inf_n w_n \geq 0$. Then
\[
0 \leq w_{n+1} \leq v_n - \chi_{n+1} + \xi_{n+1} + \sum_{k \geq n+2} \xi_k +M \leq w_n
-\chi_{n+1} \eqsp.
\]
$\sequence{w}[n][\nset]$ is non-negative and non increasing; therefore it
converges.  Furthermore, $0 \leq \sum_{k=0}^n \chi_k \leq w_0$ so that $\sum_n
\chi_n < \infty$. The convergence of $\sequence{w}[n][\nset]$ also implies the
convergence of $\sequence{v}[n][\nset]$. This concludes the proof.

\subsubsection{Proof of \autoref{prop:CvgCvx:PerturbedProximalGradient}}
\label{sec:proof:prop:CvgCvx:PerturbedProximalGradient}
Let $\theta_\star \in \L$, which is not empty by \autoref{A1:compl}; note that
$F(\theta_\star) = \min F$. We have by \eqref{eq:key1} applied with $\theta
\leftarrow \theta_n - \gamma_{n+1} H_{n+1}$, $\xi \leftarrow \theta_n$,
$\theta' \leftarrow \theta_\star$, $\gamma \leftarrow \gamma_{n+1}$
\begin{align*}
  \| \theta_{n+1} - \theta_\star \|^2 &\leq \|\theta_n - \theta_\star\|^2 - 2
  \gamma_{n+1} \left( F(\theta_{n+1}) - \min F \right) - 2 \gamma_{n+1}
  \pscal{\theta_{n+1}- \theta_\star}{\eta_{n+1}} \eqsp.
\end{align*}
We write $\theta_{n+1} - \theta_\star = \theta_{n+1} -
T_{\gamma_{n+1}}(\theta_n ) + T_{\gamma_{n+1}}(\theta_n) - \theta_\star. $ By
\autoref{lem:techlem}, $\| \theta_{n+1} -T_{\gamma_{n+1}}(\theta_n) \| \leq
\gamma_{n+1} \| \eta_{n+1}\|$ so that,
\[  - \pscal{\theta_{n+1}- \theta_\star}{\eta_{n+1}} \leq \gamma_{n+1} \|
\eta_{n+1} \|^2 - \pscal{T_{\gamma_{n+1}}(\theta_n) -\theta_\star}{\eta_{n+1}}
\eqsp.
\]
Hence,
\begin{multline}
  \label{eq:preparation:SRlemma}
  \| \theta_{n+1} - \theta_\star \|^2   \leq \|\theta_n - \theta_\star\|^2   - 2 \gamma_{n+1} \left(  F(\theta_{n+1}) - \min F\right) \\
  + 2 \gamma_{n+1}^2 \| \eta_{n+1} \|^2 - 2 \gamma_{n+1}
  \pscal{T_{\gamma_{n+1}}(\theta_n) -\theta_\star}{\eta_{n+1}} \eqsp.
\end{multline}
Under (\ref{eq:CS:BorelCantelli}) and (\ref{eq:preparation:SRlemma}),
\autoref{lemma:RobbinsSiegmund:deterministe} shows that $\sum_n \gamma_{n}
\left( F(\theta_n) -\min F \right) < \infty$ and $\lim_n \| \theta_n -
\theta_\star \|$ exists.  This implies that $\sup_n \|\theta_n \| < \infty$.
Since $\sum_n \gamma_n = +\infty$, there exists a subsequence
$\{\theta_{\phi_n}, n \in \nset \}$ such that $\lim_n F(\theta_{\phi_n}) = \min
F$. The sequence $\{\theta_{\phi_n}, n \geq 0 \}$ being bounded, we can assume
without loss of generality that there exists $\theta_\infty \in \rset^d$ such
that $\lim_n \theta_{\phi_n} = \theta_\infty$.

Let us prove that $\theta_\infty \in \L$. Since $g$ is lower semi-continuous on
$\Theta$, $\liminf_n g(\theta_{\phi_n}) \geq g(\theta_\infty)$ so that
$\theta_\infty \in \Theta$.  Since $F$ is lower semi-continuous on $\Theta$, we
have
\[
\min F= \liminf_{n \to \infty} F(\theta_{\phi_n}) \geq F(\theta_\infty) \geq \min F \eqsp,
\]
showing that $F(\theta_\infty)= \min F$.

By (\ref{eq:preparation:SRlemma}), for any $m$ and $n \geq \phi_m$
\[
\| \theta_{n+1} - \theta_\infty\|^2 \leq \|\theta_{\phi_m} - \theta_\infty\|^2
- 2 \sum_{k=\phi_m}^n \gamma_{k+1} \{ \pscal{T_{\gamma_{k+1}}(\theta_k) -
  \theta_\infty}{\eta_{k+1}} + \gamma_{k+1} \| \eta_{k+1} \|^2  \} \eqsp.
\]
For any $\epsilon >0$, there exists $m$ such that the RHS is upper bounded by
$\epsilon$. Hence, for any $n \geq \phi_m$, $\|\theta_{n+1}- \theta_\infty \|^2
\leq \epsilon$, which proves the convergence of $\sequence{\theta}[n][\nset]$
to $\theta_\infty$.

%%%%%%%%%%%%%%%%%%%%%%%%%%%%%%%%%%%%%%%%%%%%%%%%%%%%%%%%%%%
\subsubsection{Proof of \autoref{theo:rate-convergence-basic}}\label{sec:prooflemma:rate-convergence-basic}
%%%%%%%%%%%%%%%%%%%%%%%%%%%%%%%%%%%%%%%%%%%%%%%%%%%%%%%%%%%
Let $\theta_\star \in \L$; note that $F(\theta_\star) = \min F$.  We
first apply \eqref{eq:key1} with $\theta \leftarrow \theta_j-\gamma_{j+1}
H_{j+1}$, $\xi \leftarrow \theta_j$, $\theta'\leftarrow \theta_\star$, $\gamma
\leftarrow \gamma_{j+1}$:
\begin{align*}
  F(\theta_{j+1})- \min F & \leq (2\gamma_{j+1})^{-1} \,
  \left(\|\theta_j-\theta_\star\|^2-\|\theta_{j+1}-\theta_\star\|^2\right)-
  \pscal{\theta_{j+1}-\theta_\star}{\eta_{j+1}} \eqsp.
\end{align*}
Multiplying both sides by $a_{j+1}$ gives:
\begin{multline*}
  a_{j+1} \Big(F(\theta_{j+1}) -\min F \Big)\leq
  \frac{1}{2}\left(\frac{a_{j+1}}{\gamma_{j+1}}-\frac{a_{j}}{\gamma_{j}}\right)\|\theta_j-\theta_\star\|^2
  + \frac{a_{j}}{2\gamma_{j}}\|\theta_j-\theta_\star\|^2 \\
  -\frac{a_{j+1}}{2\gamma_{j+1}}\|\theta_{j+1}-\theta_\star\|^2
  -a_{j+1}\pscal{\theta_{j+1}-\theta_\star}{\eta_{j+1}}.
\end{multline*}
Summing from $j=0$ to $n-1$ gives
\begin{multline}\label{eq2:proofthm1}
  \frac{a_n}{2 \gamma_n} \|\theta_n - \theta_\star \|^2 + \sum_{j=1}^na_{j}
  \{F(\theta_{j}) - \min F \} \leq
  \frac{1}{2}\sum_{j=1}^n\left(\frac{a_j}{\gamma_j}-\frac{a_{j-1}}{\gamma_{j-1}}\right)
  \|\theta_{j-1}-\theta_\star\|^2  \\
  -\sum_{j=1}^{n} a_{j} \pscal{\theta_{j}-\theta_\star}{\eta_{j}} +
  \frac{a_0}{2\gamma_0} \|\theta_0 - \theta_\star \|^2 \eqsp.
\end{multline}
We decompose $\pscal{\theta_{j}-\theta_\star}{\eta_{j}}$ as follows:
\[
\pscal{\theta_{j}-\theta_\star}{\eta_{j}}=
\pscal{\theta_{j}-T_{\gamma_{j}}(\theta_{j-1})}{\eta_{j}}+\pscal{T_{\gamma_{j}}(\theta_{j-1})-\theta_\star}{\eta_{j}}\eqsp.
\]
By \autoref{lem:techlem}, we get
$\left|\pscal{\theta_{j}-T_{\gamma_{j}}(\theta_{j-1})}{\eta_{j}}\right|\leq
\gamma_{j}\|\eta_{j}\|^2$ which concludes the proof.

%%%%%%%%%%%%%%%%%%%%%%%%%%%%%%%%%%%%%%%%%%%%%%%%%%%%%%%%%%%
\subsection{Proof of \autoref{sec:MC:fixedabatch}}
%%%%%%%%%%%%%%%%%%%%%%%%%%%%%%%%%%%%%%%%%%%%%%%%%%%%%%%%%%%
\label{sec:proof:MC:fixedabatch}
The proof of \autoref{theo:approxsto:cvg} is given in the case $m=1$; we simply
denote by $X_n$ the sample $X_{n}^{(1)}$.  The proof for the case $m>1$ can be
adapted from the proof below, by substituting the functions $\H{\theta}{x}$ and
$W(x)$ by
\[
\overline H_{\theta}(x_1, \cdots, x_m) = \frac{1}{m} \sum_{k=1}^m \H{\theta}{x_k}
\qquad \overline W(x_1, \cdots, x_m) = \frac{1}{m} \sum_{k=1}^m W(x_k) \eqsp;
\]
the kernel $P_\theta$ and its invariant measure $\pi_\theta$ by
\begin{align*}
\overline P_\theta(x_1, \cdots, x_m; B)
&=
\idotsint  P_\theta(x_m, \rmd y_1)  \, \prod_{k=2}^{m} P_\theta(y_{k-1},\rmd y_{k}) \1_{B}(y_1,\dots,y_m) \eqsp, \\
\overline \pi_\theta(B) &= \idotsint  \pi_\theta(\rmd y_1)  \, \prod_{k=2}^{m} P_\theta(y_{k-1},\rmd y_{k}) \1_{B}(y_1,\dots,y_m) \eqsp,
\end{align*}
for any $(x_1,\dots,x_m) \in \Xset^n$ and $B \in \Xsigma^{\times n}$.
\subsubsection{Preliminary results}
\begin{proposition} 
  \label{prop:SousGradient}
  Assume that $g$ is proper convex and Lipschitz on $\Theta$ with Lipschitz
  constant $K$. Then, for all $\theta \in \Theta$,
  \begin{equation}
    \label{eq:regulariteGamma:Prox}
     \norm{ \Prox_{\gamma,g}(\theta) - \theta } \leq  K \gamma \eqsp.
  \end{equation}
\end{proposition}
\begin{proof}
For all $\theta \in \Theta$, we get by \autoref{lem:prox}
$$
0 \leq \gamma^{-1} \norm{\theta - \Prox_{\gamma,g}(\theta)}^2 \leq g(\theta) - g(\Prox_{\gamma,g}(\theta)) \leq K \norm{\theta - \Prox_{\gamma,g}(\theta)} \eqsp.
$$
\end{proof}

\begin{proposition}\label{prop:approxsto:RegulTheta:Tgamma}
  Assume \autoref{A1}, \autoref{A1:compl} and $\Theta$ is bounded. Then \[
  \sup_{\gamma \in \ocint{0,1/L}} \sup_{\theta \in \Theta} \norm{T_{\gamma}(\theta)
  } < \infty \eqsp.
  \] If in addition  \autoref{hyp:AS:withbias}-\ref{hyp:AS:StabiliteCompact:item2} holds, then  there exists a constant $C$ such
  that for any $\theta, \bar \theta \in \Theta$, $\gamma, \bar \gamma \in
  \ocint{0,1/L}$ 
\begin{align*}
  \norm{ T_\gamma(\theta) - T_{\bar \gamma}(\bar \theta)} \leq C \left( \gamma
    + \bar \gamma + \| \theta - \bar{\theta}\| \right) \eqsp.
\end{align*}
\end{proposition}
\begin{proof}
  Let $\theta_\star$ such that for any $\gamma>0$, $\theta_\star =
  T_{\gamma}(\theta_\star)$ (such a point exists by \autoref{A1:compl}
  and \eqref{eq:definition:Lset}).  We write $ T_\gamma(\theta) =
  \left(T_\gamma(\theta) - \theta_\star \right) + \theta_\star$.  By
  \autoref{lem:Lipschitz:GradientProximalMap}, there exists a constant
  $C$ such that for any $\theta \in \Theta$ and any $\gamma \in
  \ocint{0,1/L}$, $\norm{T_{\gamma}(\theta) - \theta_\star } \leq 2 \,
  \norm{\theta - \theta_\star} \leq \pma{2} \norm{\theta} + \pma{2}
  \norm{\theta_\star}$. This concludes the proof of the first
  statement. We write $T_\gamma(\theta) - T_{\bar \gamma}(\bar \theta)
  = T_\gamma(\theta) - T_{\bar \gamma}(\theta) + T_{\bar
    \gamma}(\theta) - T_{\bar \gamma}(\bar \theta)$. By
  \autoref{lem:prox}
\[
\norm{T_{\bar \gamma}(\theta) - T_{\bar \gamma}(\bar \theta)} \leq \norm{\theta
  - \bar \theta - \bar \gamma \nabla f(\theta) + \bar \gamma \nabla f(\bar
  \theta)} \leq \| \theta - \bar \theta \| + \bar \gamma \sup_{\theta \in \Theta}
\| \nabla f(\theta) \| \eqsp.
\]
By \autoref{A1} and since $\Theta$ is bounded, $ \sup_{\theta \in \Theta} \| \nabla
f(\theta) \| < \infty$.  In addition, using again \autoref{lem:prox},
\begin{align*}
  \norm{T_\gamma(\theta) - T_{\bar \gamma}(\theta)} & \leq \left(\gamma + \bar
    \gamma \right) \ \sup_{\theta \in \Theta} \|\nabla f(\theta) \| +
  \norm{\Prox_{\gamma,g}(\theta) - \Prox_{\bar \gamma,g}(\theta)} \eqsp.
\end{align*}
We conclude by using
\begin{align*}
  \norm{\Prox_{\bar \gamma,g}(\theta) - \Prox_{\gamma,g}(\theta) } & \leq
  \norm{\Prox_{\bar \gamma,g}(\theta) - \theta } + \norm{\theta - \Prox_{\gamma,g}(\theta) } \\
  & \leq \left( \gamma + \bar \gamma \right) \ \sup_{\gamma \in \ocint{0,1/L}}
  \sup_{\theta \in \Theta} \gamma^{-1} \, \norm{ \Prox_{\gamma,g}(\theta) -
    \theta } \eqsp.
\end{align*}
\end{proof}

\begin{lemma} \label{prop:Poisson}
  Assume \autoref{hyp:UniformErgo} and
  \autoref{hyp:AS:withbias}-\ref{hyp:smooth:intheta}.
  \begin{enumerate}[label=(\roman*)]
 \item \label{hyp:Poisson:hatH} There exists a measurable function $(\theta,x) \mapsto
    \hatH{\theta}{x}$ such that $\sup_{\theta \in \Theta} \normW{\hatH{\theta}{}}{W}< \infty$ and for any $(\theta,x) \in \Theta \times \Xset$,
\begin{equation}\label{eq:PoissonEquation}
\hatH{\theta}{x} - P_\theta \hatH{\theta}{x} = \H{\theta}{x} - \int \H{\theta}{y} \pi_\theta(\rmd y) \eqsp.
\end{equation}
\item \label{hyp:Poisson:RegulTheta} There exists a constant $C$ such that for
  any $\theta, \theta' \in \Theta$,
\[
\normWm{P_\theta \hatH{\theta}{} - P_{\theta'} \hatH{\theta'}{}}{W} \leq C \
\norm{\theta - \theta'} \eqsp.
\]
  \end{enumerate}
\end{lemma}
\begin{proof}
See \cite[Lemma 4.2]{fort:moulines:priouret:2012}.
\end{proof}

\begin{lemma}\label{prop:AS:moments}
  Assume \autoref{hyp:MonteCarloSamples} and \autoref{hyp:UniformErgo}. Then,
  $\sup_n \PE\left[W^p(X_n) \right] < \infty$.
\end{lemma}
\begin{proof}
  \pma{Conditionnally to the past $\F_{j-1}$, the conditional
    distribution of $X_j$ is $P_{\theta_{j-1}}(X_{j-1}, \cdot)$.
    Therefore, we write
    \[
    \PE\left[W^p(X_n) \right] = \PE\left[ \CPE{W^p(X_n)}{\F_{n-1}}
    \right] = \PE\left[ P_{\theta_{n-1}} W^p(X_{n-1})\right].
    \]
    We then use the drift inequality to obtain $\PE\left[W^p(X_n)
    \right] \leq \lambda \PE\left[ W^p(X_{n-1}) \right] + b$.  The
    proof then follows from a trivial induction. }
\end{proof}

\begin{lemma} \label{prop:approxsto:DeltaTheta}
  Assume \autoref{A1},
  \autoref{hyp:AS:withbias}-\ref{hyp:AS:StabiliteCompact:item2} and $\Theta$ is
  bounded.  There exists a constant $C$ such that w.p.1, for all $n \geq 0$,
  $$
  \norm{\theta_{n+1} - \theta_n} \leq C \gamma_{n+1} \left(1 +
    \norm{\eta_{n+1}} \right)\eqsp.
$$
\end{lemma}
\begin{proof}
  We write
\[
\theta_{n+1} - \theta_n = \theta_{n+1} - \Prox_{\gamma_{n+1},g}(\theta_n) +
\Prox_{\gamma_{n+1},g}(\theta_n) - \theta_n.
\]
Since by \autoref{lem:prox}, $\theta \mapsto \Prox_{\gamma,g}(\theta)$ is
Lipschitz for any $\gamma>0$, we get
\begin{align*}
  &\norm{\theta_{n+1} - \Prox_{\gamma_{n+1},g}(\theta_n)} =
  \norm{\Prox_{\gamma_{n+1},g}(\theta_n - \gamma_{n+1} \eta_{n+1} -\gamma_{n+1}
    \nabla
    f(\theta_n)) - \Prox_{\gamma_{n+1},g}\left( \theta_n\right)} \\
  & \quad \leq \gamma_{n+1} \norm{\eta_{n+1} + \nabla f(\theta_n)} \leq
  \gamma_{n+1} \left( \norm{\eta_{n+1}} + \sup_{\theta \in \Theta} \norm{\nabla
      f(\theta)} \right) \eqsp.
\end{align*}
By \autoref{A1}, w.p.1.  $ \sup_{\theta \in \Theta} \norm{\nabla f(\theta)} <
\infty$; hence, there exists $C_1$ such that w.p.1.  for all $n \geq 0$,
$\norm{ \theta_{n+1} - \Prox_{\gamma_{n+1},g}(\theta_n) } \leq C_1 \gamma_{n+1}
\left(1 +\norm{\eta_{n+1} } \right)$.  Finally, under \autoref{hyp:AS:withbias}-\ref{hyp:AS:StabiliteCompact:item2}, there exists a
constant $C_2$ such that, w.p.1.,
\[
\sup_n \gamma_{n+1}^{-1} \norm{ \Prox_{\gamma_{n+1},g}(\theta_n) - \theta_n }
\leq  \sup_{\gamma \in \ocint{0,1/L}} \sup_{\theta \in \Theta} \gamma^{-1} \norm{
  \Prox_{\gamma,g}(\theta) - \theta } \leq C_2 \eqsp.
\]
This concludes the proof.
\end{proof}
\begin{lemma}
  \label{prop:controle:Eta}
  Assume \autoref{A1}, \autoref{hyp:MonteCarloSamples},
  \autoref{hyp:UniformErgo} and $\Theta$ is bounded. There exists a constant $C$
  such that w.p.1, for all $n \geq 0$, $\norm{\eta_{n+1}} \leq C W(X_{n+1})$.
\end{lemma}
\begin{proof}
  By \autoref{hyp:MonteCarloSamples} and \autoref{hyp:UniformErgo},
  $\norm{\eta_{n+1}} \leq \left( \sup_{\theta \in \Theta}
    \normW{\H{\theta}{}}{W}\right) \, W(X_{n+1}) + \sup_{\theta \in \Theta} \|
  \nabla f(\theta) \|$. The result follows since $\nabla f$ is Lipschitz by
  \autoref{A1}, and since $W \geq 1$.
\end{proof}

\subsubsection{Proof of \autoref{theo:approxsto:cvg}}
The proof of the almost-sure convergence consists in verifying the
assumptions of \autoref{prop:CvgCvx:PerturbedProximalGradient}.
\pma{Let us start with the proof that almost-surely, $\sum_n
  \gamma_{n+1}^2 \| \eta_{n+1}\|^2 < \infty$. This property is a
  consequence of \autoref{prop:AS:termeEnSquare} applied with $a_n
  \leftarrow \gamma_n^2$. It remains to prove that almost-surely
  \[
  \sum_n \gamma_n \eta_n < \infty,\qquad \sum_n \gamma_{n+1}
  \pscal{T_{\gamma_{n+1}}(\theta_n)}{\eta_{n+1}} < \infty;
  \]
  note that they are both of the form $\sum_n \gamma_{n+1}
  \Am_{\gamma_{n+1}}(\theta_n) \eta_{n+1}$ with, respectively,
  $\Am_{\gamma}(\theta)$ equal to the identity matrix, and
  $\Am_{\gamma}(\theta) = T_\gamma(\theta)$.  In the case the Monte
  Carlo is unbiased, we apply \autoref{prop:AS:TermeMartingale} with
  $a_n \leftarrow \gamma_n$ and $A_\gamma(\theta)$ equal to the
  identity matrix and we obtain the almost-sure convergence of $\sum_n
  \gamma_n \eta_n$; we then apply \autoref{prop:AS:TermeMartingale}
  with $a_n \leftarrow \gamma_n$ and
  $A_\gamma(\theta)=T_\gamma(\theta)$, and we obtain the almost-sure
  convergence of $\sum_n \gamma_{n+1}
  \pscal{T_{\gamma_{n+1}}(\theta_n)}{\eta_{n+1}}$ - note that by
  \autoref{prop:approxsto:RegulTheta:Tgamma}, $T_\gamma(\theta)$
  satisfies the assumptions on $\Am_\gamma(\theta)$. In the case the
  Monte Carlo is biased, the steps are the same except we use
  \autoref{prop:AS:TermeFluctuations} instead of
  \autoref{prop:AS:TermeMartingale}.}

For the control of the moments, we use
\autoref{theo:rate-convergence-basic} and again
\autoref{prop:AS:termeEnSquare} and \autoref{prop:AS:TermeMartingale}
for the unbiased case (or \autoref{prop:AS:TermeFluctuations} for the
biased case).

\begin{lemma} \label{prop:AS:termeEnSquare}
  Assume \autoref{A1}, \autoref{hyp:MonteCarloSamples},
  \autoref{hyp:UniformErgo} and $\Theta$ is bounded.
\begin{enumerate}[label=(\roman*)]
\item If $a_k \geq 0$ and $\sum_{k=1}^\infty a_k < \infty$ then with
  probability one, $\sum_{n \geq 1} a_n \| \eta_n \|^2 < \infty$.
\item for any $q \in \ccint{1,p/2}$, there exists a constant $C$ such that for
  any non-negative numbers $\{a_1, \cdots, a_n \}$,
\[
\normLq{ \sum_{k=1}^n a_{k} \| \eta_{k} \|^2}{q} \leq C \, \sum_{k=1}^n a_k
\eqsp.
\]
\end{enumerate}
\end{lemma}
\begin{proof}
  We write
  \begin{align*}
    \PE\left[ \sum_{n \geq 0} a_{n+1} \| \eta_{n+1} \|^2 \right] \leq \sup_n
    \left( \PE\left[\| \eta_{n+1} \|^2 \right] \right) \, \sum_{n \geq 0}
    a_{n+1}\eqsp.
  \end{align*}
  By \autoref{prop:AS:moments} and \autoref{prop:controle:Eta}, $\sup_n
  \normLq{ \eta_{n+1}}{2} < \infty$ so the RHS is finite.  By the Minkovski
  inequality, we write since $a_k  >0$,
\begin{align*}
  \normLq{ \sum_{k=0}^n a_{k+1} \|\eta_{k+1} \|^2}{q} \leq \sup_n \normLq{
    \eta_{n} }{2q}^{2} \, \sum_{k=1}^{n+1} a_{k} \eqsp.
\end{align*}
The supremum is finite by \autoref{prop:AS:moments} and
\autoref{prop:controle:Eta}.
\end{proof}

\begin{proposition}
\label{prop:AS:TermeMartingale}
Assume \autoref{A1}, \autoref{A2}, \autoref{hyp:MonteCarloSamples},
\autoref{hyp:UniformErgo}, $\Theta$ is bounded and the Monte Carlo approximation is
unbiased.  Let $\sequence{a}[n][\nset]$ be a deterministic positive sequence
and $\{\Am_\gamma(\theta), \gamma \in \ocint{0,1/L}, \theta \in \Theta \}$ be
deterministic matrices such that
  \begin{align}
    & \label{eq:conditions:sur:Am:1} \sup_{\gamma \in \ocint{0,1/L}}
    \sup_{\theta \in \Theta} \| \Am_\gamma(\theta) \| < \infty \eqsp.
  \end{align}

\begin{enumerate}[label=(\roman*)]
\item If $\sum_{n \geq 0} a_n^2<\infty$, then the series $\sum_{n \geq 0}
  a_{n+1} \Am_{\gamma_{n+1}}(\theta_n) \eta_{n+1}$ converges $\PP$-\as\
\item For any $q \in \ocint{1, p/2}$, there exists a constant $C$ such that
\[
\normLq{ \sum_{k=0}^n a_{k+1} \Am_{\gamma_{k+1}}(\theta_k) \eta_{k+1} }{q} \leq
C \ \left( \sum_{k=0}^n  a_{k+1}^2\right)^{1/2} \eqsp.
\]
\end{enumerate}
\end{proposition}
\begin{proof}
  Since $\theta_n \in \F_n$, we have $\PE\left[ a_{n+1}
    \Am_{\gamma_{n+1}}(\theta_n) \ \eta_{n+1} \vert \F_n \right] = 0$, thus
  showing that $\{ M_n = \sum_{k=0}^n a_{k+1} \Am_{\gamma_{k+1}}(\theta_k)
  \eta_{k+1}, n \in \nset\}$ is a martingale. This martingale converges
  almost-surely if $S= \sum_{n \geq 0 } a_{n+1}^2 \|
  \Am_{\gamma_{n+1}}(\theta_n) \|^2 \| \eta_{n+1} \|^2 < \infty$ $\PP$-\as\
  (see e.g.~\cite[Theorem 2.17]{hall:heyde:1980}).  Using
  (\ref{eq:conditions:sur:Am:1}) and  \autoref{prop:AS:termeEnSquare}, $S < \infty$
  $\PP$-\as\

  Consider now  the $L^q$-moment of $M_n$. We apply~\cite[Theorem
  2.10]{hall:heyde:1980}: for any $q \in \ocint{1,p/2}$, there exists a
  constant $C$ such that for any $n \geq 0$,
\[
\normLq{\sum_{k=0}^n a_{k+1} \Am_{\gamma_{k+1}}(\theta_k) \eta_{k+1}}{q} \leq C
\left( \sum_{k=0}^n \normLq{ a_{k+1} \Am_{\gamma_{k+1}}(\theta_k) \eta_{k+1}
  }{q}^2\right)^{1/2} \eqsp.
\]
\autoref{prop:AS:moments} and \autoref{prop:controle:Eta} imply that $\sup_n
\normLq{\eta_{n+1}}{q} < \infty$; we then conclude with
\eqref{eq:conditions:sur:Am:1}.
\end{proof}

\begin{proposition}
\label{prop:AS:TermeFluctuations}
Assume \autoref{A1}, \autoref{A2}--\autoref{hyp:AS:withbias} and $\Theta$ is
bounded.  Let $\{a_n,n \geq 0 \}$ be a positive sequence and
$\{\Am_\gamma(\theta), \gamma \in \ocint{0,1/L}, \theta \in \Theta \}$ be
(deterministic) function-valued matrices such that there exists $C_\Am$ and for
any $\gamma, \bar \gamma \in \ocint{0,1/L}$ and $\theta, \bar{\theta} \in
\Theta$ 
\begin{equation}
\sup_{\gamma \in \ocint{0,1/L}} \sup_{\theta \in \Theta} \| \Am_\gamma(\theta)
  \| < \infty \eqsp, \qquad \norm{\Am_\gamma(\theta) - \Am_{\bar
      \gamma}(\bar{\theta})} \leq C_\Am \left( \gamma + \bar \gamma +
    \norm{\theta- \bar{\theta}} \right) \eqsp. \label{eq:conditions:sur:Am:2}
\end{equation}
\begin{enumerate}[label=(\roman*)]
\item If $\sum_n a_n \gamma_n < \infty $, $\sum_n a_n^2 < \infty$ and $\sum_n
  |a_{n+1} -a _n | < \infty$ then the series $\sum_{n \geq 0} a_{n+1}
  \Am_{\gamma_{n+1}}(\theta_n) \eta_{n+1}$ converges $\PP$-\as\ 
\item For any $q \in \ocint{1, p/2}$, there exists a constant $C$ such that
\begin{multline*}
  \normLq{ \sum_{k=0}^n a_{k+1} \Am_{\gamma_{k+1}}(\theta_k) \eta_{k+1} }{q}
  \leq C \ \left\{ 1 + \left( \sum_{k=0}^n a_{k+1}^2\right)^{1/2} +
    \sum_{k=1}^n \left| a_{k+1} -a_k \right| + \sum_{k=1}^n a_k \gamma_k
  \right\} \eqsp.
\end{multline*}
\end{enumerate}
\end{proposition}
\begin{proof}
\begin{enumerate}[label=(\roman*), wide=0pt, labelindent=\parindent]
\item By \autoref{hyp:MonteCarloSamples} and
  \autoref{prop:Poisson}-\ref{hyp:Poisson:hatH}, we write
\begin{align*}
  \eta_{n+1} & = \hatH{\theta_n}{X_{n+1}} - P_{\theta_n}
  \hatH{\theta_n}{X_{n+1}} \\
  & = \left( \hatH{\theta_n}{X_{n+1}} - P_{\theta_n} \hatH{\theta_n}{X_{n}}
  \right) + \left( P_{\theta_n} \hatH{\theta_n}{X_{n}} - P_{\theta_{n+1}}
    \hatH{\theta_{n+1}}{X_{n+1}}  \right) \\
  & + \left( P_{\theta_{n+1}} \hatH{\theta_{n+1}}{X_{n+1}} - P_{\theta_n}
    \hatH{\theta_n}{X_{n+1}}\right) \eqsp.
\end{align*}
We prove successively that w.p.1,
\begin{eqnarray}
\label{eq:proof:approxsto:step1}
\sum_n a_{n+1} \Am_{\gamma_{n+1}}(\theta_n) \
 \left( \hatH{\theta_n}{X_{n+1}} - P_{\theta_n} \hatH{\theta_n}{X_{n}} \right) < \infty \eqsp, \\
\label{eq:proof:approxsto:step2a}
\sum_{n \geq 0} a_{n+1} \Am_{\gamma_{n+1}}(\theta_n) \ \left(
P_{\theta_n} \hatH{\theta_n}{X_{n}} - P_{\theta_{n+1}}
  \hatH{\theta_{n+1}}{X_{n+1}} \right) < \infty \eqsp, \\
\label{eq:proof:approxsto:step2b} \sum_{n \geq 0} a_{n+1} \Am_{\gamma_{n+1}}(\theta_n) \left(
P_{\theta_{n+1}} \hatH{\theta_{n+1}}{X_{n+1}} - P_{\theta_n}
  \hatH{\theta_n}{X_{n+1}} \right)  < \infty \eqsp.
\end{eqnarray}
{\em Proof of~(\ref{eq:proof:approxsto:step1})}
By~\autoref{hyp:MonteCarloSamples}, $\{ \hatH{\theta_n}{X_{n+1}} - P_{\theta_n}
\hatH{\theta_n}{X_{n}}, n \in \nset\}$ is a martingale increment w.r.t. the
filtration $\{\F_n, n \geq 0\}$. The proof is along the same lines as the proof
of \autoref{prop:AS:TermeMartingale} upon noting that by \autoref{prop:Poisson}
and \autoref{hyp:UniformErgo}, there exists $C$ such that w.p.1 for all $n \geq
0$,
\[
\| \hatH{\theta_n}{X_{n+1}} - P_{\theta_n} \hatH{\theta_n}{X_{n}}\| \leq C \,
\left( W(X_{n+1}) + W(X_n)\right) \eqsp.
\] 
{\em Proof of (\ref{eq:proof:approxsto:step2a})} The sum is equal to $\sum_{n
  \geq 0} \Delta_{n+1} P_{\theta_n} \hatH{\theta_n}{X_{n}}$ with $\Delta_{n+1}
= a_{n+1} \Am_{\gamma_{n+1}}(\theta_n) - a_{n} \Am_{\gamma_{n}}(\theta_{n-1})$.
On one hand, by \autoref{prop:Poisson} and \autoref{hyp:UniformErgo}, there
exists $C$ such that w.p.1 for all $n \geq 0$,
\[
\norm{P_{\theta_n} \hatH{\theta_n}{X_{n}}} \leq C \, W(X_n) \eqsp.
\]
On the other hand, by \eqref{eq:conditions:sur:Am:2},
\autoref{prop:approxsto:DeltaTheta} and \autoref{prop:controle:Eta}, there
exists $C$ such that a.s. 
\[
\text{for all  $n \geq 0$}, \qquad \norm{\Delta_{n+1}} \leq C \Big(\left| a_{n+1} -
    a_n\right| + a_n \left( \gamma_n + \gamma_{n+1} \right) \Big) W(X_{n})
\eqsp.
\]
By \autoref{prop:AS:moments}, $\sup_n \PE\left[ W^2(X_n) \right] < \infty$.
Therefore, by (\ref{eq:conditions:sur:Am:2}) and the assumptions on $\{a_n,n
\geq 0\}$, we have $\sum_n \PE\left[ \left\| \Delta_{n+1} \ P_{\theta_n}
    \hatH{\theta_n}{X_{n}} \right\|\right] < \infty$; which concludes the
proof.

{\em Proof of (\ref{eq:proof:approxsto:step2b})} By
(\ref{eq:conditions:sur:Am:2}) and \autoref{prop:Poisson}, there exists a
constant $C$ such that w.p.1  for any $n$
\[
\left\| \Am_{\gamma_{n+1}}(\theta_n) \left( P_{\theta_{n+1}}
    \hatH{\theta_{n+1}}{X_{n+1}} - P_{\theta_n} \hatH{\theta_n}{X_{n+1}}
  \right)\right\| \leq C \, \norm{\theta_{n+1} - \theta_n} \, W(X_{n+1}) \eqsp.
\]
By \autoref{prop:approxsto:DeltaTheta} and \autoref{prop:controle:Eta}, there
exists a constant $C$ such that w.p.1,  
\[
\text{for all $n \geq 0$}, \qquad \norm{\theta_{n+1} - \theta_n} W(X_{n+1})
\leq C \gamma_{n+1} \, W^2(X_{n+1}) \eqsp.
\]
From \autoref{prop:AS:moments} and the assumptions on $\{a_n, n \geq 0\}$,
$\sum_n a_{n+1} \, \gamma_{n+1} \PE\left[W^2(X_{n+1}) \right] < \infty$ from
which (\ref{eq:proof:approxsto:step2b}) follows.

\item We start from the same decomposition of
$\eta_{n+1}$ in three terms. The first one is a martingale, and following the
same lines as in the proof of \autoref{prop:AS:TermeMartingale}, we obtain
\[
\normLq{\sum_{k=0}^n a_{k+1} \Am_{\gamma_{n+1}}(\theta_n) \ \left(
    \hatH{\theta_n}{X_{n+1}} - P_{\theta_n} \hatH{\theta_n}{X_{n}} \right) }{q}
\leq C \, \left( \sum_{k=0}^n a_{k+1}^2 \right)^{1/2} \eqsp.
\]
For the second term, we write
\begin{align*}
  \sum_{k=0}^n a_{k+1} & \Am_{\gamma_{k+1}}(\theta_k) \ \left( P_{\theta_k}
    \hatH{\theta_k}{X_{k}} - P_{\theta_{k+1}} \hatH{\theta_{k+1}}{X_{k+1}}
  \right) \\
  & \leq a_1 \Am_{\gamma_{1}}(\theta_0) P_{\theta_0} \hatH{\theta_0}{X_{0}} -
  a_{n+1} \Am_{\gamma_{n+1}}(\theta_n) P_{\theta_{n+1}}
  \hatH{\theta_{n+1}}{X_{n+1}} \\
  & + \sum_{k=1}^n \Delta_{k+1} \ P_{\theta_k} \hatH{\theta_k}{X_{k}} \eqsp.
\end{align*}
By the Minkovski inequality, it is easily seen that there exists a constant $C$
such that 
\begin{multline*}
  \normLq{\sum_{k=0}^n a_{k+1} \Am_{\gamma_{k+1}}(\theta_k) \ \left(
      P_{\theta_k} \hatH{\theta_k}{X_{k}} - P_{\theta_{k+1}}
      \hatH{\theta_{k+1}}{X_{k+1}}
    \right)}{q} \\
  \leq \left(1 + a_{n+1} + \sum_{k=1}^n \Big( \left| a_{k+1} - a_k \right| +
    a_k \left( \gamma_k + \gamma_{k+1} \right) \Big)\right) \eqsp.
\end{multline*}
Finally, for the last term, following the same computations as above, we have
by the Minkovski inequality 
\[
\normLq{\sum_{k=0}^n a_{k+1} \Am_{\gamma_{k+1}}(\theta_k) \ \left(
    P_{\theta_{k+1}} \hatH{\theta_{k+1}}{X_{k+1}} - P_{\theta_k}
    \hatH{\theta_k}{X_{k+1}} \right) }{q} \leq C \sum_{k=0}^n a_{k+1}
\gamma_{k+1} \eqsp.
\]
\end{enumerate}
\end{proof}

\subsection{Proof of \autoref{theo:IncreasingBatch}}
\label{theo:proof:IncreasingBatch}
We write $\eta_{n+1} = B_n + \left( \eta_{n+1} - B_n \right)$
\pma{where $B_n$ is given by \eqref{eq:MonteCarlo:bias}.}  Observe that $\{\eta_{n+1} - B_n, n
\in \nset \}$ is a martingale-increment sequence.  Sufficient
conditions for the almost-sure convergence of a martingale and the
control of $L^q$-moments can be found in ~\cite[Theorems 2.10 and
2.17]{hall:heyde:1980}. Then the proof follows from
\autoref{prop:erreurLp:sumMC} and \autoref{prop:AS:moments}.

\appendix

\section{\autoref{sec:ex2}}
\label{app:ex_fun}
By using the Cauchy-Schwartz inequality, it holds
\begin{align*}
  & \int \exp(\ell_c(\theta\vert \bu)) \, \phi(\bu) \rmd \bu \geq \left( \int
    \exp(0.5\ell_c(\theta\vert
    \bu)) \,  \phi(\bu) \rmd \bu\right)^{1/2} \\
  & \left( \int \exp(\ell_c(\theta\vert \bu)) \, \|u\|^2 \, \phi(\bu) \, \rmd \bu
  \right)^2 \leq \left( \int \exp(0.5\ell_c(\theta\vert \bu)) \phi(\bu) \, \rmd
    \bu\right) \ \left( \int \exp(3 \ell_c(\theta\vert \bu)/2) \|u\|^4
    \phi(\bu) \rmd \bu\right)
\end{align*}
which implies that
\begin{align*}
  \int \|u\|^2 \pi_\theta(\bu) \rmd \bu & = \frac{ \int \exp(\ell_c(\theta\vert
    \bu)) \|u\|^2 \phi(\bu) \rmd \bu}{\int \exp(\ell_c(\theta\vert v)) \phi(v)
    \rmd v}  \\
  & \leq \left(\int \exp\left(3 \ell_c(\theta \vert\bu) /2 \right) \, \| \bu
    \|^4 \, \phi(\bu) \, \rmd \bu\right)^{1/2}
\end{align*}
Since $\exp(\ell_c(\theta \vert \bu)) \leq 1$ (it is the likelihood of i.i.d.
Bernoulli variables) and $\int \|\bu\|^4 \phi(\bu) \rmd \bu =q(2+q)$, we have
\[
\sup_{\theta \in \Theta} \int \|u\|^2 \pi_\theta(\bu) \rmd \bu \leq \sqrt{q(2+q)} \eqsp.
\]

\section{ \autoref{sec:ex1}}\label{app:example:network}
For $\theta,\vartheta\in\Theta$, the $(i,j)$-th entry of the matrix
$\nabla\ell(\theta)-\nabla\ell(\vartheta)$ is given by
\[\left(\nabla\ell(\theta)-\nabla\ell(\vartheta)\right)_{ij}=\int_{\Xset^p} \bar B_{ij}(x)\pi_\vartheta(\rmd x)-\int_{\Xset^p} \bar B_{ij}(x)\pi_\theta(\rmd x).\]
For $t\in  \ccint{0,1}$ let
\[\pi_t(dz)\eqdef \exp\left(\pscal{\bar B(z)}{t\vartheta+(1-t)\theta}\right)/\int \exp\left(\pscal{\bar B(x)}{t\vartheta+(1-t)\theta}\right)\mu(\rmd x),\]
 defines a probability measure on $\Xset^p$. It is straightforward to check that
\[\left(\nabla\ell(\theta)-\nabla\ell(\vartheta)\right)_{ij}=\int \bar B_{ij}(x)\pi_1(\rmd x)-\int \bar B_{ij}(x)\pi_0(\rmd x),\]
 and that $t\mapsto \int \bar B_{ij}(x)\pi_t(\rmd x)$ is differentiable with derivative
\begin{eqnarray*}
\frac{\rmd}{\rmd t}\int \bar B_{ij}(x)\pi_t(\rmd x)&=&\int \bar B_{ij}(x )\pscal{\bar B(x)-\int \bar B(z)\pi_t(\rmd z)}{\vartheta-\theta}\pi_t(\rmd x),\\
&=& \textsf{Cov}_{\pi_t}\left(\bar B_{ij}(X),\pscal{\bar B(X)}{\vartheta-\theta}\right),\end{eqnarray*}
where the covariance is taken assuming that $X\sim\pi_t$. Hence
\begin{align*}
\left|\left(\nabla\ell(\theta)-\nabla\ell(\vartheta)\right)_{ij}\right|
%% &=\left|\int_0^1 \rmd t\left[\frac{\rmd}{\rmd t}\int B_{ij}(x_i,x_j)\pi_t(\rmd x)\right]\right|\\
&=\left|\int_0^1 \rmd t \ \ \textsf{Cov}_t\left(\bar B_{ij}(X),\pscal{\bar B(X)}{\vartheta-\theta}\right)\right|\\
&\leq \textsf{osc}(\bar B_{ij})\sqrt{\sum_{k\leq l }\textsf{osc}^2(\bar B_{kl})} \|\theta-\vartheta\|_2.
\end{align*}
This implies the inequality (\ref{eq:propExample}).

{\bf Acknowledgments:} We are grateful to George Michailidis for very helpful
discussions. This work is partly supported by NSF grant DMS-1228164.

\bibliographystyle{ims}
\bibliography{biblio}

\end{document}